\documentclass[paper=a4, fontsize=11pt]{scrartcl}

\usepackage[T1]{fontenc}
\usepackage[utf8]{inputenc}
\usepackage[english]{babel}
\usepackage{babelbib}
\usepackage{amsmath,amsfonts,amsthm,amssymb}
\usepackage[left=2cm,right=2cm,top=2.85cm,bottom=2.85cm]{geometry}

\usepackage[Algorithm]{algorithm}
\usepackage{algorithmic}
\usepackage{graphicx}
\usepackage{setspace}

\DeclareMathOperator*{\argmin}{arg\,min}

\usepackage{xcolor}

\providecommand{\keywords}[1]{\textbf{\textit{Key words ---}} #1}
\providecommand{\MSC}[1]{\textbf{\textit{MSC ---}} #1}

\newcommand{\norm}[1]{\left\lVert#1\right\rVert}

\newcommand{\cl}[1]{\operatorname{cl}\!\left(#1\right)}

\newtheorem{theorem}{Theorem}[section]
\newtheorem{lemma}[theorem]{Lemma}
\newtheorem{korollar}[theorem]{Corollary}
\newtheorem{definition}[theorem]{Definition}

\newtheorem{assumption}{Assumption}

\numberwithin{equation}{section}

\usepackage{hyperref}
\DeclareMathAlphabet{\mathcall}{OMS}{cmsy}{m}{n}

\newcommand{\R}{\mathbb{R}}
\newcommand{\N}{\mathbb{N}}

\newcommand{\Hk}{\mathcall{H}_k(\Omega)}
\newcommand{\Hkdual}{\mathcall{H}_k(\Omega)'}
\newcommand{\HkR}{\mathcall{H}_k\left(\mathbb{R}^N\right)}
\newcommand{\HkdualR}{\mathcall{H}_k\left(\mathbb{R}^N\right)'}
\newcommand{\normW}[2]{ \left\|#1\right\|_{#2}}
\newcommand{\inte}[4]{ \int_{#1}^{#2} \! \left [ #4 \right ] \text{d}#3}
\newcommand{\inteNo}[4]{ \int_{#1}^{#2} \!   #4  \text{d}#3}

\title{Recovery of the optimal control value function in reproducing kernel Hilbert spaces from verification conditions}

\begin{document}
\selectlanguage{english}
\title{Recovery of the optimal control value function in reproducing kernel Hilbert spaces from verification conditions}
  \author{Tobias Ehring \thanks{Institute of Applied Analysis and Numerical Simulation, University of Stuttgart, Pfaffenwaldring 57, Stuttgart 70569, Baden-W\"urttemberg, Germany. E-mail: ehringts@mathematik.uni-stuttgart.de} \and  Behzad Azmi\thanks{Department of Mathematics and Statistics, University of Konstanz, Universit\"atsstra\ss e 10, Konstanz 78457, Baden-W\"urttemberg, Germany. E-mail: behzad.azmi@uni-konstanz.de} \and Bernard Haasdonk\thanks{Institute of Applied Analysis and Numerical Simulation, University of Stuttgart, Pfaffenwaldring 57, Stuttgart 70569, Baden-W\"urttemberg, Germany. E-mail: haasdonk@mathematik.uni-stuttgart.de}}
  \date{\today}
  \maketitle
\begin{abstract}
\noindent Approximating the optimal value function \(v^*\) for infinite-horizon, nonlinear, autonomous optimal control problems is both challenging and essential for synthesizing real-time optimal feedback. We develop an abstract optimal recovery framework in reproducing kernel Hilbert spaces (RKHS) for reconstructing unknown target functions from mixed equality and inequality functional constraints.  Within this framework, the approximation of \(v^*\) is cast as a collocation-type problem derived from verification conditions for optimality -- most prominently, the Hamilton-Jacobi-Bellman (HJB) equation -- that uniquely characterizes \(v^*\). As the set of collocation points becomes dense in the ambient domain \(\Omega\), we establish convergence of the RKHS approximants to \(v^*\): globally on \(\Omega\) in the RKHS norm when \(v^*\) is analytic, and locally (in a neighborhood of the origin) in the RKHS norm when \(v^*\) is bounded from above and below by quadratic functions. Furthermore, we show that a practical numerical realization of the abstract scheme reduces to the classical policy iteration algorithm.  Numerical experiments support the effectiveness of the proposed approach.
   \end{abstract}
   
\noindent \keywords{Value function approximation, Reproducing kernel Hilbert spaces, Hamilton-Jacobi-Bellman equation, infinite-horizon optimal controls, Collocation methods}

\noindent \MSC{49L20,  46E22,  41A29, 49M05, 49N35}

 \section{Introduction}   
   Designing optimal feedback controllers is essential to ensure robust real-time regulation of dynamic systems despite uncertainties, disturbances, and modeling inaccuracies. These controllers adjust inputs in response to the system’s current state, steering its trajectory to minimize a specified cost. Here, we study an infinite-horizon optimal control problem (OCP) with quadratic control contribution in the running cost. The objective functional is given by
\[
J_{\infty}(x,\mathbf{u}):= \inte{0}{\infty}{t}{h\!\left(\mathbf{x}_{\mathbf{u}}(t;x)\right)+  \left \langle \mathbf{u}(t) , R \mathbf{u}(t)\right\rangle },
\]
with the convention \(J_{\infty}(x,\mathbf{u})=+\infty\) whenever the corresponding state trajectory exhibits finite-time blow-up. The state trajectories are generated by the control-affine system
\begin{gather}
\dot{\mathbf{x}}_{\mathbf{u}}(t;x)=f\!\left(\mathbf{x}_{\mathbf{u}}(t;x)\right)+g\!\left(\mathbf{x}_{\mathbf{u}}(t;x)\right)\mathbf{u}(t)\quad\text{and}\quad \mathbf{x}_{\mathbf{u}}(0;x)=x\in\mathbb{R}^N,\label{eq:ODEInfinit}
\end{gather}
where bold symbols denote time-dependent scalar or vector functions, thereby distinguishing them from time-independent quantities. The admissible input set is defined as
\[
\mathcal{U}_\infty:=L^\infty([0,\infty);\mathbb{R}^M)=\Bigl\{\mathbf{u}:[0,\infty)\to\mathbb{R}^M~\big|~\mathbf{u}\text{ measurable and essentially bounded}\Bigr\}.
\]
Then, the corresponding OCP reads
\begin{equation}\tag{OCP}\label{eq:MPInfinit}
v^*(x) := \inf_{\mathbf{u}\in\mathcal{U}_\infty} J_{\infty}(x,\mathbf{u}),
\end{equation}
where the function \(v^*(x)\) is called the optimal value function (OVF). It returns the minimal achievable cost-to-go starting from the initial condition  \(x\). For inputs \(\mathbf{u}\in\mathcal{U}_\infty\), the solutions of \eqref{eq:ODEInfinit} are understood in the Carathéodory sense \cite{filippov1988differential}; later, we restrict attention to continuous feedbacks induced by a continuously differentiable OVF, in which case trajectories are classical solutions of \eqref{eq:ODEInfinit}. Throughout the analysis, we impose the following conditions on the system data, which, in particular, ensure that $h(x) + \langle u, R u \rangle $ is nonnegative.
\begin{assumption}\label{as:data}
Let \(f\in C^1(\mathbb{R}^N,\mathbb{R}^N)\), \(g\in C^1(\mathbb{R}^N,\mathbb{R}^{N\times M})\), \(h\in C^1(\mathbb{R}^N,\mathbb{R})\), and let \(R\in\mathbb{R}^{M\times M}\) be symmetric positive definite (i.e., \(\left \langle u ,Ru \right\rangle   >0\) for all \(u\in\mathbb{R}^M\setminus\{0\}\)). Assume \(f(0)=0\) and that \(h\) is positive definite (i.e., \(h(x)>0\) for all \(x\neq 0\) and \(h(0)=0\)). 
\end{assumption}
\noindent For clarity, $\langle\cdot,\cdot\rangle$ denotes the Euclidean inner product and
$\|\cdot\|$ the corresponding Euclidean norm on  $\mathbb{R}^{N}$. Moreover, for any symmetric positive definite matrix
$A\in\mathbb{R}^{d\times d}$ and $x \in \mathbb{R}^d$, we write
$
  \normW{x}{A} \;:=\; \sqrt{\langle x,Ax\rangle}.
$ \\

\noindent Assumption~\ref{as:data} implies that \(f\) and \(g\) are locally Lipschitz in the state, so for every \(x\in\mathbb{R}^N\) and every \(\mathbf{u}\in\mathcal{U}_\infty\) the initial value problem \eqref{eq:ODEInfinit} admits a unique maximal Carathéodory solution on some interval \([0,T_{\max})\) with $T_{\max} \in (0,\infty]$. Moreover, if \(v^*(x)<\infty\), then any minimizing sequence \((\mathbf{u}_k)_{k\in\mathbb{N}}\subset\mathcal{U}_\infty\) with \(J_{\infty}(x,\mathbf{u}_k)\to v^*(x)\) may, without loss of generality, be taken to avoid finite-time escape, because such trajectories are assigned infinite cost by convention; consequently, every trajectory associated with an $\mathbf{u}_k$ is uniquely defined for all \(t\ge 0\).\\

\noindent By Bellman’s principle of optimality (see Proposition 2.5, Chapter 3 of \cite{bardi1997optimal}), the infinite-horizon problem \eqref{eq:MPInfinit} is equivalent, for each \(T>0\), to the family of finite-horizon problems
\begin{gather}
v^*(x)=\inf_{\mathbf{u}\in\mathcal{U}_{T}}\!\left\{J_T(x,\mathbf{u})+v^*\!\left(\mathbf{x}_{\mathbf{u}}(T;x)\right)\right\},\label{eq:MPFinitFirst}\\
\dot{\mathbf{x}}_{\mathbf{u}}(t;x)=f\!\left(\mathbf{x}_{\mathbf{u}}(t;x)\right)+g\!\left(\mathbf{x}_{\mathbf{u}}(t;x)\right)\mathbf{u}(t) \quad \text{and}\quad \mathbf{x}_{\mathbf{u}}(0;x)=x\in\mathbb{R}^N\label{eq:ODEFiniteFirst}
\end{gather}
with admissible inputs \(\mathcal{U}_T:=L^\infty\!\left([0,T];\mathbb{R}^M\right)\) and running cost
\[
J_T(x,\mathbf{u}):= \inte{0}{T}{t}{h\!\left(\mathbf{x}_{\mathbf{u}}(t;x)\right)+ \normW{\mathbf{u}(t)}{R}^2},
\]
where \(J_T(x,\mathbf{u})=+\infty\) whenever the associated trajectory exhibits finite-time escape. This formulation reflects the property that any segment of an optimal trajectory is itself optimal. Assuming that the OVF $v^*$ is continuously differentiable on a domain $\Omega \subset \mathbb{R}^N$ containing the origin, one can derive a partial differential equation (PDE) that $v^*$ fulfills.
Heuristically, by rearranging \eqref{eq:MPFinitFirst}--\eqref{eq:ODEFiniteFirst}, dividing by $T$ and letting $T$ go to $0$, it follows that $v^*$ satisfies the PDE
\begin{align}
 \min_{u \in \mathbb{R}^M}\left\lbrace \left\langle f(x)+g(x)u , \nabla v^*(x) \right\rangle + h(x) + \normW{u}{R}^2 \right\rbrace = 0  \label{eq:HBJ}
\end{align}
for all $x \in \Omega$ and $v^*(0) = 0$; for a rigorous derivation, see \cite[Chapter 3]{bardi1997optimal}. Equation \eqref{eq:HBJ} is the so-called Hamilton-Jacobi-Bellman (HJB) equation. Since the minimization in \eqref{eq:HBJ} is convex in $u$, the first-order optimality condition yields the explicit, unique minimizer
\begin{align} 
\mathcal{K}(x;\nabla v^*) := u^*(x) = -\frac{1}{2}R^{-1}g(x)^{\top} \nabla v^*(x). \label{eq:feedback}
\end{align}
Substituting \eqref{eq:feedback} into \eqref{eq:HBJ} gives the equivalent form
\begin{align}\label{eq:HBJ1}
 \left\langle f(x) , \nabla {v^*}(x) \right\rangle - \frac{1}{4} \normW{g(x)^{\top} \nabla v^*(x)}{R^{-1}}^2   +h(x) &= 0 \\
 v^*(0) &= 0. \label{eq:HBJ2}
\end{align}
Moreover, inserting the feedback control law \eqref{eq:feedback} as a controller in \eqref{eq:ODEInfinit} results in the optimal closed-loop system with the optimal control signal being a function of the state. Consequently, the problem of optimal feedback control is solved once the OVF is available.
However, computing the OVF directly from the HJB equation \eqref{eq:HBJ1} is challenging, as the HJB is a fully nonlinear first‑order PDE. Moreover, the equation can have multiple solutions. To isolate the solution that coincides with the OVF, we later impose additional conditions that ensure uniqueness. These are the so-called verification conditions for optimality. On the basis of these conditions, we formulate an infinite-dimensional collocation-type minimization problem in reproducing kernel Hilbert spaces (RKHS). We prove that this infinite-dimensional formulation is equivalent to a finite-dimensional constrained optimization problem, and that the corresponding approximants converge to the OVF as the set of collocation points becomes dense in $\Omega$.\\

\noindent In the sequel, we survey several classes of numerical schemes that have been proposed to approximate the OVF.\\
\textbf{Finite–difference discretizations.} The most classical approach consists of a finite–difference (FD) approximation of the HJB; see, for example, \cite{crandal1984two,abgrall1996numerical,shu2007high}. Under suitable monotonicity and consistency conditions, these schemes provably converge to the viscosity solution, a generalized notion of solution. Unfortunately, their computational complexity grows exponentially with the dimension of the state space, making FD methods impractical beyond two or three dimensions.\\
\textbf{Semi–Lagrangian schemes.} A second, widely used family of methods starts with a temporal discretization of the controlled ordinary differential equation (ODE) that underlies the OCP, leading to the so-called semi–Lagrangian (SL) schemes. The foundational analyses in \cite{dolcetta1983discrete,falcone1994discrete} -- and the monograph \cite{falcone2013semi} -- establish the convergence of SL discretizations to the viscosity solution, even when the OVF fails to be differentiable, at least in some discounted-cost setting.
For practical computations, one still requires a finite-dimensional surrogate of the OVF. This is achieved by selecting an interpolant that satisfies the temporally discrete HJB equation at the grid nodes while permitting evaluation at off-grid points. Variants of SL schemes therefore differ mainly in the underlying interpolation space: sparse grids \cite{bokanowski2013adaptive}, radial basis functions \cite{alla2023hjb}, and polynomial approximations on adaptive tree-structured grids \cite{grune1997adaptive} are representative examples.\\
\textbf{Dimensionality reduction.} All grid-based discretizations -- finite-difference and SL alike -- remain fundamentally constrained by the curse of dimensionality. A prototypical mitigation strategy is model-order reduction followed by the application of a grid-based HJB solver to the reduced system, as demonstrated in \cite{kunisch2004hjb}.\\
\textbf{Data–driven OVF approximation.}
Recent research has focused increasingly on approximating the OVF from data generated by open‐loop optimal controls obtained by Pontryagin's Maximum Principle.  
Surrogates have been constructed with deep neural networks \cite{Nakamura21}, sparse‐grid polynomials \cite{azmi2021optimal}, and kernel methods \cite{EHRING2022325,ehring2024hermite}.  
In \cite{ehring2024hermite}, two authors of the present article designed the surrogate to enforce positive semidefiniteness rigorously. 
The data-driven approximation can be made online-adaptive and combined with model-order reduction, as illustrated in \cite{ehring2025online}. 
Across all three of our studies \cite{EHRING2022325,ehring2024hermite,ehring2025online}, as well as in the present work, we employ kernel methods, since they are generally more robust to the curse of dimensionality and, under suitable regularity and sampling assumptions, can exhibit convergence rates with respect to the number of interpolation points that have a decay factor that is independent of the ambient dimension \cite{wenzel2023analysis}. 
Moreover, kernel surrogates are grid-free and admit a rigorous functional-analytic treatment within RKHSs, a perspective that we also leverage in the current work.
\\
\textbf{Model predictive control (MPC).}
 Unlike the offline approximation approach discussed above, model predictive control (MPC)~\cite{gruene2011} approximates the optimal feedback of the infinite-horizon OCP by repeatedly solving finite-horizon OCPs online at each sampling instant, while applying the resulting control for a very short time horizon. Classical analyses investigate how large the prediction horizon 
$T$ must be in order to guarantee closed-loop stability; see, e.g., \cite{grune2009analysis,Rebele2012uncostrained,azmi2016stabilizability}.
\\
\textbf{Fixed–point schemes.}
Value iteration (VI)~\cite{bellman1957B} is an iterative scheme derived to solve the HJB in the SL form that guarantees global convergence under standard hypotheses~\cite{falcone1987}. However, its convergence rate typically degrades as the spatial grid is refined. Computational efficiency can be improved by approximating the OVF with surrogate models such as Shepard moving-least-squares interpolants based on radial basis functions~\cite{alla2023hjb} or neural networks~\cite{heydari2014}. Another widely‐employed fixed-point scheme is the policy iteration (PI) method, originally introduced in \cite{bellman1957}.  As the practical numerical realization of the approximation scheme for the OVF that we propose later in this work is algorithmically equivalent to PI, we provide a more detailed exposition of PI:\\
For the PI, at each iteration, given a current feedback law \(u\colon \Omega\to\mathbb{R}^M\) with $u(0) = 0$, the procedure comprises two stages: \\
\textit{1. Policy Evaluation:} Compute the value function \(v_u\colon \Omega\to\mathbb{R}\) associated with the current policy \(u\) by solving the linearized Hamilton-Jacobi-Bellman equation, also known as the generalized HJB (GHJB) equation:
\begin{equation}\label{eq:GHJB}
  \mathrm{GHJB}(v_u,u,x)
  :=
  \bigl\langle f(x)+g(x)\,u(x),\,\nabla v_u(x)\bigr\rangle
  + h(x)
  + \normW{u(x)}{R}^2 
  = 0,
  \quad
  v_u(0)=0,
  \quad
  x\in\Omega.
\end{equation}
The solution \(v_u\) to the GHJB is given by the infinite‐horizon performance index of \(u\), i.e.,
\begin{equation*}
  v_u(x)
  =
  \inte{0}{\infty}{t}{
      h\bigl(\mathbf{x}_u(t;x)\bigr)
      + \normW{u(\mathbf{x}_u(t;x))}{R}^2 
    },
\end{equation*}
where the closed‐loop trajectory \(\mathbf{x}_u(t;x)\) satisfies
\begin{equation*}
  \dot{\mathbf{x}}_u(t;x)
  = f\bigl(\mathbf{x}_u(t;x)\bigr)
    + g\bigl(\mathbf{x}_u(t;x)\bigr)\,u\bigl(\mathbf{x}_u(t;x)\bigr),
  \quad
  \mathbf{x}_u(0;x)=x,
\end{equation*}
if the triple
\((u,\Omega,v_u)\) is {admissible}, meaning
finite cost (\(v_u(x) < \infty\) for every \(x \in \Omega\)),  forward invariance (\(\mathbf{x}_u(t;x) \in \Omega\) for all \(t \ge 0\)
        whenever \(x \in \Omega\)) and regularity (\(v_u \in C^{1}(\Omega,\mathbb{R})\)) must be satisfied.\\
\textit{2. Policy Improvement:} Update the control law by minimizing the Hamiltonian with respect to \(u\), yielding
\begin{equation*}
  u^{+}(x)
  =
  -\tfrac{1}{2}\,R^{-1}\,g(x)^{\top}\,\nabla v_u(x).
\end{equation*}

\noindent These two steps are alternated until convergence, i.e., until the pair \(\bigl(u,v_u\bigr)\) satisfies the full HJB equation \eqref{eq:HBJ1}-\eqref{eq:HBJ2}. Existing numerical methods differ mainly in how the linear GHJB in each step is solved. Low-rank approximation techniques, such as those in~\cite{eigel2023dynamical,dolgov2021tensor}, efficiently represent the value function; Galerkin approximation approaches~\cite{kalise2018polynomial,beard1997galerkin} minimize an \(L^{2}\)-residual of the generalized HJB equation, providing a polynomial framework for approximating the value function; neural network approximations \cite{abu2005nearly,Kamalapurkar2018,vamvoudakis2010online}, and grid-based schemes enhanced with model-order reduction~\cite{alla2020B} have also been employed. \\

\noindent Nevertheless, most PI-based schemes either lack rigorous convergence guarantees or rely on restrictive assumptions. The root difficulty arises even for {exact} PI (i.e., without numerical approximation of the GHJB solution): the well-posedness of the GHJB equation and the selection of the iteration domain \(\Omega\) are delicate issues; see~\cite{ehring2025convergencepolicyiterationinfinitehorizon} for a detailed analysis. Introducing additional discretization and approximation layers only worsens these challenges. Motivated by these considerations, we adopt a different strategy: we construct approximations of the OVF directly from verification conditions for optimality. This leads to a finite-dimensional optimization problem whose solutions provably approximate \(v^*\). Moreover, we show that a practical iterative solver for this program is algorithmically equivalent to a PI scheme.\\

\noindent We proceed by describing the main contributions and the overall organization of the article.
In Section 2, we derive verification conditions for the OVF on possibly bounded domains by adapting standard arguments from dynamic programming and HJB theory.  These conditions are crucial for recovering the OVF within an RKHS.
In Section 3, we review kernel methods and formulate linear and nonlinear optimal recovery problems. Furthermore, we extend the nonlinear optimal recovery framework proposed in \cite{chen2021solving} to accommodate functional inequality constraints. Moreover, we provide rigorous conditions under which these problems admit finite-dimensional representations and show that a practical solution approach via a Gauss-Newton ansatz is equivalent to a linear optimal recovery problem. In Section 4, we combine the verification conditions with the RKHS-based nonlinear optimal recovery framework to obtain the collocation-type scheme for approximating the OVF. We establish convergence of the resulting approximations to the OVF as the collocation points become dense in the ambient domain and, via a Gauss–Newton ansatz, show that the resulting practical iterative solver for the collocation problem is equivalent to PI with a standard generalized kernel interpolation ansatz for solving the GHJB.
In Section 5, we present numerical studies on four model problems, including a 50-dimensional nonlinear OCP.  We conclude with a discussion and outlook, highlighting open questions in Section 6.

\section{Verification conditions for the OVF}\label{sec:VerificationCon}

Constructing an approximation scheme for the OVF solely from the HJB equation is problematic, since the HJB provides only a necessary condition for the OVF. In particular, the equation can have multiple solutions -- even in the standard linear–quadratic setting -- making it non‑trivial to identify the OVF \(v^*\) by solving the HJB alone.
As an illustration, consider the two-dimensional problem satisfying Assumption~\ref{as:data} with
\[
h(x_1,x_2) \;=\; x_1^2 + x_2^2,
\quad
R \;=\;
\begin{bmatrix}
1 & 0\\
0 & 1
\end{bmatrix},
\quad
f(x_1,x_2) \;=\;
\begin{bmatrix}
x_1\\
x_2
\end{bmatrix},
\quad
g(x_1,x_2) \;=\;
\begin{bmatrix}
1 & 0\\
0 & 1
\end{bmatrix}.
\]
The associated algebraic Riccati equation (ARE) arising from the  HJB equation~\eqref{eq:HBJ1} under the quadratic OVF ansatz
\(v(x)=x^\top P x\) with \(P=P^\top\in\mathbb{R}^{2\times 2}\) is
\[
2P - P^2 + I_2 \;=\; 0,
\]
where \(I_2\) denotes the \(2\times 2\) identity matrix. In this instance, the ARE admits infinitely many real symmetric solutions of the form $P= Q\operatorname{diag}\bigl(\lambda_1,\lambda_2\bigr)Q^{\top}$ with $\lambda_1,\lambda_2 \in \{ 1+\sqrt{2},1-\sqrt{2} \}$ for any orthogonal matrix $Q \in \mathbb{R}^{2 \times 2}$.  
Only \(P^* = \operatorname{diag}\bigl(1+\sqrt{2},\,1+\sqrt{2}\bigr) \) is positive‐definite, and thus
\[
v^*(x) \;=\; \langle x, P^* x \rangle
\]
is the unique valid OVF. All other solutions generate candidate OVFs that are negative along some directions towards the origin and therefore do not represent the optimal cost-to-go, which must be nonnegative.
A necessary condition for any candidate OVF is therefore positive semidefiniteness, i.e.,
\[
v(x)\ge 0 \quad\forall x\in\Omega, 
\qquad
v(0)=0.
\]
However, even this requirement together with satisfying the HJB equation are not sufficient for a general set $\Omega$, because also \(\tilde v(x)=\langle x, \tilde P x \rangle\) with $\tilde{P}=\operatorname{diag}\bigl(1+\sqrt{2},\,1-\sqrt{2}\bigr)$ is nonnegative on the non‑trivial cone
\[
\tilde M:=\bigl\{x\in\mathbb{R}^2 \mid \langle x, \tilde P x \rangle\ge 0\bigr\}.
\]
We must therefore enforce the HJB equation and positive semidefiniteness on a domain -- a nonempty, open, connected subset of \(\mathbb{R}^N\) -- that contains the origin.  Note that \(\tilde M\) is not a domain in this sense, as it does not contain a neighborhood of the origin.
With this observation we obtain the following local results.

\begin{theorem}[Local verification of optimality]\label{thm:VerOpt}
Assume that Assumption~\ref{as:data} holds and let
\(\Omega\subset\mathbb{R}^{N}\) be a domain containing the origin.
Suppose the OVF $v^*\in C^1(\Omega,\mathbb{R})$. If there exists a candidate $v\in C^{1}(\Omega,\mathbb{R})$ such that
\begin{enumerate}
  \item[(1)] \(v\) is positive semidefinite (\(v(x)\ge 0\) for all \(x\in\Omega\) and \(v(0)=0\));
  \item[(2)] the HJB identity holds:
        \[
          \bigl\langle f(x)+g(x)u_{v}(x),\nabla v(x)\bigr\rangle
          + h(x) + \normW{u_{v}(x)}{R}^2 = 0,
          \qquad\forall\,x\in\Omega,
        \]
        where
        \[
          u_{v}(x) := -\tfrac12\,R^{-1} g(x)^{\!\top}\nabla v(x),
        \]
\end{enumerate}
then there exists a set $\tilde{\Omega}\subset\Omega$ containing a neighborhood of the origin such that
\[
  v(x) \;=\; v^{*}(x)\qquad \forall\,x\in\tilde{\Omega}.
\]
In particular, $u_v$ is an optimal feedback on $\tilde{\Omega}$.
\end{theorem}
\begin{proof}
    See Appendix \ref{sec:appoptVer}.
\end{proof}
\noindent A limitation of Theorem~\ref{thm:VerOpt} is that the set $\tilde{\Omega}$ depends on the particular candidate $v$ and, in general, need not coincide with the domain $\Omega$ on which the HJB equation is satisfied. This dependence is inconvenient for the convergence analysis of the verification-based approximation scheme for the OVF presented in Section 4. To overcome this issue, we seek additional sufficient conditions guaranteeing that the candidate OVF is equal to the OVF on a fixed domain, which is independent of the specific candidate. Our first result in this direction exploits real analyticity. Let $A(\Omega)$ denote the space of real-analytic functions on $\Omega$. Combining Theorem~\ref{thm:VerOpt} with the identity theorem for analytic functions yields the following corollary.

\begin{korollar}\label{coro:VerificAnalytic}
In addition to the hypotheses of Theorem~\ref{thm:VerOpt}, suppose that $v^{*}\in A(\Omega)$ and
\begin{enumerate}
  \item[(3)] $v\in A(\Omega)$.
\end{enumerate}
Then
\[
  v(x)\equiv v^{*}(x) \qquad \forall\,x\in\Omega .
\]
\end{korollar}

\begin{proof}
See Appendix~\ref{sec:appoptVerAna}.
\end{proof}
\noindent Real analyticity is a strong assumption -- particularly for the OVF -- and is typically hard to verify for a given OCP. Nevertheless, it yields equality on the entire domain $\Omega$. 
The next corollary replaces analyticity with alternative conditions that are substantially easier to certify; the resulting conclusion guarantees equality only on a neighborhood of the origin, but this neighborhood is independent of the particular candidate.
\begin{korollar}\label{coro:VerificConstant}
In addition to the hypotheses of Theorem~\ref{thm:VerOpt}, suppose there exist constants $\alpha,\beta>0$ such that
\begin{enumerate}
  \item[(3)] \(\alpha\,\lVert x\rVert^{2}\ \le\ v(x)\ \le\ \beta\,\lVert x\rVert^{2}\quad \text{for all } x\in\Omega.\)
\end{enumerate}
Then there exists  an open neighborhood of the origin $\tilde{\Omega}\subset\Omega$, depending only on $\Omega$, $\alpha$, $\beta$, and $v^{*}$, with
\[
  v(x)=v^{*}(x) \qquad \forall\,x\in\tilde{\Omega}.
\]
Moreover, the set $\tilde{\Omega}$ is independent of the particular candidate $v$.
\end{korollar}
\begin{proof}
    See Appendix \ref{sec:VerificConstant}.
\end{proof}
\noindent 
A principal practical challenge in applying Corollary~\ref{coro:VerificConstant} is the selection of the constants $\alpha>0$ and $\beta>0$. For certain OCP instances and  an ill-chosen pair $(\alpha,\beta)$, the feasibility set may be  empty: no candidate function $v$ can satisfy conditions (1)–(2) of Theorem~\ref{thm:VerOpt} together with condition (3) of Corollary~\ref{coro:VerificConstant}. 
Necessary conditions for the OVF $v^*$ to satisfy quadratic two-sided bounds are naturally expressed in terms of stabilizability and detectability properties of the linearization at the origin. In essence, suppose $h$ and $v^*$ are twice continuously differentiable in a neighborhood of the origin, and let
\[
A := Df(0), 
\qquad B := g(0), 
\qquad Q := \tfrac{1}{2}\,H_{h}(0),
\]
and assume that $(A,B)$ is stabilizable and $(Q^{1/2},A)$ is detectable with  $Df$ denoting the Jacobian of $f$ and $H_{h}$ the Hessian of $h$. Then the continuous-time ARE associated with the linearized dynamics,
\begin{equation*}
  A^{\top} P \;+\; P\,A
  \;-\; P\,B\,R^{-1} B^{\!\top} P
  \;+\; Q \;=\; 0,
\end{equation*}
admits a unique, symmetric, positive definite solution $P\in\mathbb{R}^{N\times N}$; see  \cite{lancaster1995algebraic}. Under these conditions, one obtains local quadratic bounds: there exists a neighborhood $\Omega_{0}$ of the origin such that
\[
    \tfrac{1}{2}\,\lambda_{\min}(P)\,\lVert x\rVert^{2}\ \le\ v^{*}(x)\ \le\ 2\,\lambda_{\max}(P)\,\lVert x\rVert^{2}
    \qquad \text{for all } x\in\Omega_{0},
\]
which follows from the  expansion $v^{*}(x)=x^{\top}P x + o(\lVert x\rVert^{2})$ as $x\to 0$. In our numerical experiments in Section~5, these lower/upper quadratic constants proved to be adequate on the full computational domain $\Omega$, even though the theoretical guarantee above is local. For a bounded set $\Omega$ (denoting with $\operatorname{cl}$  the closure of a set), a standard globalization  is to set
\[
\alpha^{*}
  := \min\!\Bigl\{
        \tfrac12\,\lambda_{\min}(P),\;
        \min_{x\in \cl{\Omega\setminus \Omega_{0}}}
          \frac{v^{*}(x)}{\|x\|^{2}}
      \Bigr\},
\qquad
\beta^{*}
  := \max\!\Bigl\{
       2\,\lambda_{\max}(P),\;
        \max_{x\in \cl{\Omega\setminus \Omega_{0}}}
          \frac{v^{*}(x)}{\|x\|^{2}}
      \Bigr\},
\]
which are finite by continuity of $x\mapsto v^{*}(x)/\|x\|^{2}$ away from the origin. A detailed analysis, while standard, would require additional notation and is omitted here. Motivated by these considerations, we adopt the following standing assumption.

\begin{assumption}[Quadratic regularity at the origin]\label{as:data2}
For the given OCP, there exist constants $\alpha^*,\beta^* >0$ such that
\[
\alpha^*\,\lVert x\rVert^{2}\ \le\ v^*(x)\ \le\ \beta^*\,\lVert x\rVert^{2}
\qquad \text{for all } x\in\Omega.
\]
\end{assumption}

\noindent Building on the verification conditions established above, we develop two collocation-type schemes in RKHSs: one for real-analytic OVFs and one for OVFs that satisfy global quadratic bounds. The construction and analysis of these schemes are presented in Section~4, after we introduce an abstract optimal recovery framework in RKHSs that provides the theoretical foundation.

\section{Background on kernel methods}
We begin by recalling basic notions of positive definite kernels and RKHSs. 
Let $\Omega$  be a nonempty set and consider a symmetric function  
$
k:\Omega\times\Omega\to\mathbb{R}
$
called a kernel.  It is said to be {positive definite} (p.d.) if for every finite set of pairwise distinct points  
$
X_n:=\{x_1,\dots,x_n\}\subset\Omega,
$  
the Gram matrix  
$
\mathcal{K}_{X_n}:=\bigl(k(x_i,x_j)\bigr)_{i,j=1}^{n}
$  
is positive semidefinite.  If $\mathcal{K}_{X_n}$ is positive definite for every $X_n$, then $k$ is called {strictly} positive definite (s.p.d.).  Clearly, every s.p.d.\ kernel is p.d.
Positive definite kernels are particularly important because each of them induces a unique RKHS $\Hk$ of real‑valued functions on~$\Omega$ with $k$ as its reproducing kernel (see \cite{wendland2004}).  A Hilbert space $\Hk$ of functions $\xi:\Omega\to\mathbb{R}$ is an RKHS with reproducing kernel $k$ if  
 $k(x,\cdot)\in\Hk$ for all $x\in\Omega$, and
   the {reproducing property}
     \begin{equation}\label{eq:repProp}
        \langle \xi,\,k(x,\cdot)\rangle_{\Hk} = \xi(x),\qquad\forall \xi\in\Hk,\;x\in\Omega,
       \end{equation}
        holds.
This property also extends to derivatives. Let \(\Omega \subset \mathbb{R}^N\) be a nonempty open set, and let \(k \in C^{2}(\Omega \times \Omega;\mathbb{R})\). Then, for every \(x \in \Omega\), every \(\xi \in \mathcal{H}_k(\Omega)\), and every \(s \in \{1,\dots,N\}\),
\begin{equation}\label{eq:repPropDeriv}
  \partial^{s}\xi(x) = \bigl\langle \xi, \partial^s_{1} k(x,\cdot) \bigr\rangle_{\mathcal{H}_k(\Omega)} ,
\end{equation}
where \(\partial^s_{1}\) denotes the partial derivative with respect to the \(s\)-th coordinate of the first argument of \(k\).
  This is a direct consequence of the reproducing property and the smoothness of~$k$; see  ~\cite[Theorem~10.45]{wendland2004}.

\subsection{Linear optimal recovery problems in RKHSs}
\label{subsec:lin_opt_rec}
The verification–based methods that we employ for approximating the OVF are rooted in the theory of {nonlinear optimal recovery} problems.
Because these nonlinear problems are a direct extension of their linear counterparts -- often called generalized kernel interpolation -- we first summarize the linear theory.  
Our exposition follows \cite[Chap.~16]{wendland2004}, but is adapted to the requirements of the present study. Although generalized kernel interpolation constitutes a broad methodological framework, its most prominent applications arise in the numerical solution of linear PDEs, resulting in a kernel collocation method. As a running example, we consider the linear GHJB equation \eqref{eq:GHJB}; beyond illustrating the concepts, this example enables a kernel-based PI scheme in which each policy-evaluation step reduces to solving the collocation-type ansatz for the GHJB equation.\\
\noindent A {linear optimal recovery} (or minimal‑norm interpolation) problem reads:
\begin{definition}[Linear optimal recovery in an RKHS]
Let $k:\Omega\times\Omega\to\R$ be a p.d. kernel with RKHS $\Hk$. 
Given continuous linear functionals $\lambda_1,\dots,\lambda_n\in\Hk'$ and a target vector $r=(r_1,\dots,r_n)^\top\in\R^n$. 
The {linear optimal recovery} problem consists in finding the generalized minimal-norm interpolant
\begin{equation}\label{eq:genInterpol}
  s^n \;=\; \argmin_{s\in\Hk}\Bigl\{\,\|s\|_{\Hk} \;\big|\; \lambda_j(s)=r_j \text{ for } j=1,\dots,n \Bigr\}.
\end{equation}
\end{definition}
\noindent If the functionals $\lambda_1,\dots,\lambda_n$ are linearly independent, \cite[Theorem~16.1]{wendland2004} guarantees that problem \eqref{eq:genInterpol} admits a unique minimizer.  
Moreover, this minimizer possesses a finite‑dimensional representation
\begin{equation*}
  s^{n}(x)\;=\;
  \sum_{j=1}^{n}\alpha_j\,w_j(x),
  \qquad
  w_j\;:=\;\mathcal{R}_{\mathcal{H}_k}\lambda_j\;\in\;\mathcal{H}_k(\Omega),
\end{equation*}
where $\mathcal{R}_{\mathcal{H}_k}\colon\mathcal{H}_k(\Omega)'\to\mathcal{H}_k(\Omega)$ denotes the Riesz map, which maps linear functionals to their corresponding Riesz representers.  
The coefficients $\underline{\alpha}=\begin{bmatrix}
    \alpha_1 & \cdots & \alpha_n  
\end{bmatrix}^\top\in\mathbb{R}^n$ are obtained by enforcing the generalized interpolation conditions $\lambda_i\!\bigl(s^{n}\bigr)=r_i$ for $i=1,\dots,n$, which leads to the linear system
\[
  \mathcal{K}_{\Lambda}\,\underline{\alpha}= \underline{r},
  \qquad
  (\mathcal{K}_{\Lambda})_{ij}= \langle w_i , w_j \rangle_{\mathcal{H}_k(\Omega)},
  \qquad
  \underline{r}=\begin{bmatrix}
    r_1 & \cdots & r_n  
\end{bmatrix}^\top\in\mathbb{R}^n,
\]
where $\mathcal{K}_{\Lambda}$ is the {generalized Gramian} matrix associated with $\Lambda=\begin{bmatrix}
    \lambda_1 & \cdots & \lambda_n  
\end{bmatrix}\in \left(\Hkdual\right)^n$ and is symmetric positive definite.
Furthermore, the minimal achievable norm is
\begin{equation}
  \label{eq:min_norm_value}
  \min_{s\in\mathcal{H}_k(\Omega)}
  \Bigl\{\,\|s\|_{\mathcal{H}_k(\Omega)} \;\Big|\;
        \lambda_j(s)=r_j,\; j=1,\dots,n
  \Bigr\}
  \;=\;
 \sqrt{ \left\langle \underline{r}, \mathcal{K}_{\Lambda}^{-1}\underline{r} \right\rangle}
\end{equation}
under the condition that functionals $\lambda_1,\dots,\lambda_n$ are linearly independent.\\
When applying this ansatz to approximately solve the GHJB equation in~\eqref{eq:GHJB}, we obtain the following problem: 
For a p.d. kernel $k \in C^{2}(\Omega \times \Omega,\mathbb{R})$ and a finite set of pairwise distinct points $x_0:=0$ and $X_{n}=\{x_{1},\ldots,x_{n}\}\subset\Omega \setminus \{0 \}$ (indexing starts at zero), the goal is to find a surrogate $s_{v_u}^{n}$ such that
\begin{align}
   \langle f(x_{j})+g(x_{j})u(x_{j}), \nabla s_{v_u}^{n}(x_{j}) \rangle
      &= -\,h(x_{j}) - \normW{u(x_{j})}{R}^2 && j=1,\ldots,n, \label{eq:minNormGHJB1}\\
   s_{v_u}^{n}(0) &= 0. \label{eq:minNormGHJB2}
\end{align}
Thus the target values are
\begin{align*}
   r_{j} &= -\,h(x_{j})  - \normW{u(x_{j})}{R}^2, \\
   r_{0} &= 0,
\end{align*}
and the corresponding linear functionals are given by
\begin{align*}
        \lambda_{j}(\,\cdot\,)
      &= \left \langle f(x_{j})+g(x_{j})u(x_{j}), \delta_{x_{j}}   \circ \nabla (\,\cdot\,) \right\rangle, & j&=1,\ldots,n, \label{eq:LinFunc}\\
   \lambda_{0}(\,\cdot\,) &= \delta_{0}(\,\cdot\,),
\end{align*}
where $\delta_{x_{j}}$ denotes the point‑evaluation functional.
These operators are continuous linear functionals in $\mathcal{H}_{k}(\Omega)'$, as shown in the next lemma.
\begin{lemma}
  \label{lem:point-and-grad-eval}
  Let $\Omega\subset\mathbb{R}^{N}$ be a nonempty open set and let
  $k\in C^{2}(\Omega\times\Omega,\mathbb{R})$ be a p.d. kernel.
  \begin{enumerate}
    \item For every $x\in\Omega$ the point‑evaluation functional
       $
        \delta_{x}(\xi)=\xi(x)
     $ for $\xi \in \Hk$,
      belongs to $\mathcal{H}_{k}(\Omega)'$ with Riesz representer  $k(x,\cdot)\in\mathcal{H}_{k}(\Omega)$.
    \item 
      For each $x\in\Omega$ and $a \in \mathbb{R}^N$ the functional
      $
        \lambda_{x,a}(\xi)= \langle a,\nabla \xi(x)\rangle
      $ for $\xi \in \Hk$,
      belongs to $\mathcal{H}_{k}(\Omega)'$ with Riesz representer $w_{x,a} = \langle a, \nabla_1 k(x,\cdot) \rangle $. Here, $\nabla_{1}$ denotes the gradient with respect to the first
      argument of $k$.
  \end{enumerate}
\end{lemma}

\begin{proof}
 1.  By the reproducing property, we have  
      $
        \delta_{x}(\xi)=\xi(x)=\langle \xi,\,k(x,\cdot)\rangle_{\mathcal{H}_{k}}.
      $
      The Cauchy-Schwarz inequality gives
      \(
        |\xi(x)|
        \le\|\xi\|_{\mathcal{H}_{k}}\|k(x,\cdot)\|_{\mathcal{H}_{k}},
      \)
      establishing continuity and the claimed representer.\\
2. Because $k\in C^{2}(\Omega\times\Omega,\mathbb{R})$, \eqref{eq:repPropDeriv} can be applied.
      Hence
      \[
        \lambda_{x,a }(\xi)
        =\sum_{s=1}^{N}\left(a\right)_s\partial^{s}\xi(x)
        =\Bigl\langle \xi,\sum_{s=1}^{N}\left(a\right)_s\partial^{s}_{1}k(x,\cdot)\Bigr\rangle_{\Hk}
        =\langle \xi,w_{x,a}\rangle_{\Hk}.
      \]
      Again, Cauchy-Schwarz yields
      \(
        |\lambda_{x,a }(\xi)|
        \le\|\xi\|_{\Hk}\|w_{x,a }\|_{\Hk},
      \)
      proving continuity and identifying the representer $w_{x,a}$.
\end{proof}
\noindent  Thus, in the GHJB setting, the  Riesz representers are
\[
  w_{0}(\cdot)=k(0,\cdot),\qquad
  w_{j}(\cdot)
  =\langle f(x_{j})+g(x_{j})u(x_{j}), \nabla_{1}k(x_{j},\cdot) \rangle,
  \quad j=1,\dots,n.
\]
Furthermore, for $i,j\ge 1$, the  entries of the  {generalized Gramian} matrix  admit the explicit expression
\[
  \langle w_{j},w_{i}\rangle_{\mathcal{H}_{k}}
  = \left\langle f(x_{i})+g(x_{i})u(x_{i}), 
    \,\mathcal{E}_{k}(x_{i},x_{j})\,
    \left(f(x_{j})+g(x_{j})u(x_{j})\right)\right\rangle,
\]
with the mixed Hessian
\[
  \mathcal{E}_{k}(x_{i},x_{j})
  :=
  \nabla_2 \nabla_1 k(x_i,x_j) \in \mathbb{R}^{N \times N }.
\]

For $i=0,\;j\ge 1$, we have
\[
  \langle w_{j},w_{0}\rangle_{\mathcal{H}_{k}}
  =\left\langle f(x_{j})+g(x_{j})u(x_{j}) , \nabla_{1}k(x_{j},0)\right\rangle,
\]
and for $i=j=0$,
\[
  \langle w_{0},w_{0}\rangle_{\mathcal{H}_{k}}=k(0,0).
\]
A   necessary condition  for the unique solvability of the interpolation problem is the
nonsingularity of the Gram matrix $\mathcal{K}_{\Lambda}$, which, as noted at the beginning of this subsection, is equivalent to the linear independence of the functionals $\{\lambda_j\}_{j=0}^{n}$. To ensure this property, we impose the following structural assumption on the kernel $k$:
\begin{assumption}\label{ass:LinearInd}
Let $\Omega \subset \mathbb{R}^N$ be a nonempty open set and let $k:\Omega\times\Omega\to\mathbb{R}$ be s.p.d. with RKHS $\mathcal{H}_k(\Omega)$.
Assume that the family
\begin{align}\label{eq:linearInWend}
\bigl\{\delta_x(\,\cdot\,)\bigr\}_{x\in\Omega}
\;\cup\;
\bigl\{\delta_x(\,\cdot\,) \circ\partial^{s}(\,\cdot\,) \bigr\}_{x\in\Omega,\ s=1,\dots,N}
\end{align}
is linearly independent in the dual space $\mathcal{H}_k(\Omega)'$.
\end{assumption}
\noindent By \cite[Theorem~16.4]{wendland2004}, Assumption~\ref{ass:LinearInd} is satisfied, for example, when $k$ is an  s.p.d. kernel on $\mathbb{R}^N$ of the form
$k(x,x')=\phi(x-x')$ with $\phi\in L^{1}(\mathbb{R}^{N},\mathbb{R})\cap C^{2}(\mathbb{R}^{N})$.
The next lemma gives a sufficient condition for linear independence
of the functionals that arise in collocation methods for the GHJB equation, as well as in methods based on the verification conditions in Section 4.
\begin{lemma}\label{theo:KernelLinFunc}
Let the kernel $k:\Omega\times\Omega\to\mathbb{R}$ satisfy Assumption~\ref{ass:LinearInd}, let $\ell\in\mathbb{N}$, and let
$A:\Omega\to\mathbb{R}^{N\times \ell}$ be a matrix-valued function with $\operatorname{rank}A(x)=\ell$ and columns $A_j(x)$ for $j=1,...,\ell$. Further, define the functionals
\[
D_{x,j}(\xi)\;:=\; \left\langle A_j(x),\nabla \xi(x)\right\rangle, \qquad \xi\in \Hk,\quad j=1,\dots,\ell.
\]
Then,  the family
\begin{align}\label{eq:setFunctionals}
 \{\delta_x(\,\cdot\,)\}_{x\in \Omega}\ \cup\ \{D_{x,j}(\,\cdot\,)\}_{x\in \Omega,\ j=1,\dots,\ell}
\end{align}
is linearly independent in $\mathcal{H}_k(\Omega)'$.
\end{lemma}
\begin{proof}
By assumption
$
    \bigl\{\delta_{x}(\,\cdot\,)\bigr\}_{x\in\Omega}
    \cup
    \bigl\{\delta_{x}(\,\cdot\,)  \circ \partial^{s}(\,\cdot\,)\bigr\}_{x\in\Omega,\,s=1,\dots,N}
$
forms a linearly independent family in $\mathcal{H}_k(\Omega)'$.
Assume, for contradiction, that \eqref{eq:setFunctionals} is linearly
dependent. Then there exist distinct points
$
    \{x_{1},\dots,x_{n}\}\subset \Omega,$ $
    \{\tilde{x}_{1},\dots,\tilde{x}_{\tilde{n}}\}\subset \Omega
$
and a nonzero vector
\[
    \bigl[a^{\top},\,\tilde{a}_1^{\top},\dots,\tilde{a}_{\tilde{n}}^{\top}\bigr]^{\top}
    \in \mathbb{R}^{\,n+\ell \cdot \tilde{n}}
\]
such that
\[
   0 = \sum_{i=1}^{n} a_{i}\,\delta_{x_{i}}(\,\cdot\,)
    +
    \sum_{q=1}^{\tilde{n}} \sum_{j=1}^{\ell} (\tilde{a}_{q})_{j}\,D_{\tilde{x}_q,j}(\,\cdot\,)
    = \sum_{i=1}^{n} a_{i}\,\delta_{x_{i}}(\,\cdot\,)
    +
    \sum_{q=1}^{\tilde{n}} \sum_{s=1}^{N}  (b_{q})_{s}\,\delta_{\tilde{x}_q}(\,\cdot\,)  \circ\partial^{s}(\,\cdot\,),
    \quad
    b_{q}:= A(\tilde{x}_q)\,\tilde{a}_q .
\]
This contradicts the linear independence in \eqref{eq:linearInWend}, since,
by the rank condition $\operatorname{rank} A(x)=\ell$ for all $x \in \Omega$,
we have
\[
\bigl[a^{\top},\,\tilde{a}_1^{\top},\dots,\tilde{a}_{\tilde{n}}^{\top}\bigr]^{\top}\neq 0
\quad\Longrightarrow\quad
\bigl[a^{\top},\,b_{1}^{\top},\dots,b_{\tilde{n}}^{\top}\bigr]^{\top}\neq 0 .
\]
Therefore, \eqref{eq:setFunctionals} is linearly independent.
\end{proof}

\noindent The results presented so far provide a rigorous framework for implementing
PI in RKHSs as a
numerical method for approximating the OVF. 
The numerical approximation ansatz introduced later is algorithmically 
equivalent to executing PI in an RKHS, where the policy-evaluation step is realized by solving the generalized interpolation problem described above. 
The complete procedure is summarized in Algorithm~\ref{algo:PI}, which we henceforth refer to as {RKHS–PI}.
\begin{algorithm}
\caption{RKHS policy iteration}\label{algo:PI}
\begin{flushleft}
\textbf{Input:}\\[2pt]
\hspace*{1.5em}initial feedback \(u_{0}\);\; convergence tolerance \(\varepsilon>0\); \;  set of pairwise distinct centers $X_n=\{x_1,...,x_n\} \subset \Omega \setminus \{0\}$\\[6pt]

\textbf{Initialisation:}\\[2pt]
\hspace*{1.5em}%
\(s_{v_{-1}}^n(x)\equiv 0,\;
e_{0}:=\varepsilon+1,\;
\eta:= 0\) \\[6pt]

\textbf{Main loop:}\\[2pt]
\textbf{while} \(e_{\eta}>\varepsilon\)\ \textbf{do}\\
\hspace*{1.5em}%
\textbf{(1)}\; solve the generalized minimal-norm interpolation problem
\eqref{eq:minNormGHJB1}–\eqref{eq:minNormGHJB2} on $X_n$ for $s_{v_\eta}^n(x)$;
\\[4pt]
\hspace*{1.5em}%
\textbf{(2)}\; update the feedback law
$
u_{\eta+1}(x)\;:=\;-\tfrac12\,R^{-1}g(x)^{\!\top}\nabla s_{v_\eta}^n(x);
$\\[4pt]
\hspace*{1.5em}%
\textbf{(3)}\; compute the residual  
\(e_{\eta+1}\;:=\; \max_{x \in X_n} \left| s_{v_\eta}^n(x)-s_{v_{\eta-1}}^n(x)\right|;\)\\[4pt]
\hspace*{1.5em}%
\textbf{(4)}\; increment the counter \(\eta\;:=\;\eta+1\).\\
\textbf{end while}\\[6pt]

\textbf{Output:}\\[2pt]
\hspace*{1.5em}%
approximate OVF \(s_{v_{\eta-1}}^n\).
\end{flushleft}
\end{algorithm}
\noindent A natural question is whether Step~\textbf{(1)} is well-posed at every
iteration, i.e., whether the associated generalized Gramian matrix (built from the
linear functionals appearing in \eqref{eq:minNormGHJB1}–\eqref{eq:minNormGHJB2})
is nonsingular. Equivalently, the involved functionals must be linearly
independent on the RKHS. The next result provides sufficient conditions
ensuring that Algorithm~\ref{algo:PI} is well defined.

\begin{theorem}[Well-posedness of the RKHS policy-evaluation step]\label{theo:WellPI}
Suppose Assumption~\ref{as:data} holds and the kernel
\(k:\Omega\times\Omega\to\mathbb{R}\) satisfies
Assumption~\ref{ass:LinearInd}. Consider a set of pairwise distinct centers 
\(X_{n}=\{x_{1},\dots,x_{n}\}\subset\Omega\setminus\{0\}\) and the
corresponding initial feedback values
\(\{u_{0}(x_{1}),\dots,u_{0}(x_{n})\}\subset\mathbb{R}^{M}\).
If
\[
  f(x_{i})+g(x_{i})\,u_{0}(x_{i}) \;\neq\; 0,
  \qquad i=1,\dots,n,
\]
then, for every iteration index \(\eta\ge 0\), the generalized minimal-norm
interpolation problem in Step~\textbf{(1)} of Algorithm~\ref{algo:PI}
(see \eqref{eq:minNormGHJB1}–\eqref{eq:minNormGHJB2}) admits a unique solution.
\end{theorem}
\begin{proof}
We show that, for all \(\eta\in\mathbb{N}_0\),
\begin{equation}\label{eq:setFunctionals22}
\bigl\{\delta_{0}(\,\cdot\,)\bigr\}\cup
\Bigl\{
\left\langle f(x_{i})+g(x_{i})u_{\eta}(x_{i}), \delta_{x_{i}}  \circ \nabla (\,\cdot\,) \right\rangle
\Bigr\}_{i=1}^{n}
\subset\Hkdual
\end{equation}
is linearly independent.  
Invoking Lemma~\ref{theo:KernelLinFunc} with $\ell = 1$ for all $i=1,..,n$, it suffices to verify
\begin{align}\label{eq:nonNull}
A(x_i) := f(x_{i})+g(x_{i})u_{\eta}(x_{i})\neq 0
\qquad\text{for } i=1,\dots,n.
\end{align}
The case \(\eta=0\) is ensured by the hypothesis.  
Suppose, to the contrary, that there exist \(\eta>0\) and \(i\in\{1,\dots,n\}\)
such that \(f(x_{i})+g(x_{i})u_{\eta}(x_{i})=0\).
Then \(f(x_{i})=-g(x_{i})u_{\eta}(x_{i})\), and the GHJB equation gives
\begin{align*}
-\,h(x_{i})- \normW{u_{\eta-1}(x_{i})}{R}^2 &= \left\langle f(x_{i})+g(x_{i})u_{\eta-1}(x_{i}),\nabla s_{v_{\eta-1}}^{\,n}(x_{i})\right\rangle\\
 &=
\left\langle g(x_{i})(u_{\eta-1}(x_{i})-u_{\eta}(x_{i})),\nabla s_{v_{\eta-1}}^{\,n}(x_{i})\right\rangle.
\end{align*}
Using \(u_{\eta}(x)=-\tfrac12 R^{-1}g(x)^{\!\top}\nabla s_{v_{\eta-1}}^{\,n}(x)\)
and completing the square yields
\begin{align*}
  2\langle u_{\eta}(x_{i})-u_{\eta-1}(x_{i})), R u_{\eta}(x_{i}) \rangle
       &= -\,h(x_{i}) - \normW{ u_{\eta-1}(x_{i})}{R}^2 \\
\Longrightarrow \normW{u_{\eta}(x_{i})}{R}^2+ \normW{u_{\eta}(x_{i})-u_{\eta-1}(x_{i})}{R}^2
&=-\,h(x_{i})<0,
\end{align*}
which contradicts the positive definiteness of \(h\) and the non-negativity
of norms.  Hence  \eqref{eq:nonNull} holds, implying that
\eqref{eq:setFunctionals22} is linearly independent.  
\end{proof}

\noindent Although the well-posedness of Algorithm~\ref{algo:PI} is guaranteed under mild assumptions (cf.\ Theorem~\ref{theo:WellPI}), its convergence analysis is far from straightforward. As noted in the introduction, even with an {exact} solution of
the GHJB equation in each
policy-evaluation step, convergence is highly nontrivial; see the sufficient
conditions in~\cite{ehring2025convergencepolicyiterationinfinitehorizon}.
The additional approximation error introduced by the kernel-collocation
approach further complicates the analysis and cannot  be managed just by the standard arguments proposed in
\cite{beard1995improving,beard1997galerkin,bea1998successive}. \\
\noindent This motivates an alternative analytical viewpoint: we reexamine the algorithm through the lens of the verification-based approximation scheme for the OVF. To that end, we first develop the abstract framework of {nonlinear optimal recovery} in RKHSs next.

\subsection{Nonlinear optimal recovery problems in RKHSs}\label{sec:nonlinearOp}
\noindent The verification conditions for the OVF established in Section~2 comprise both {equality} constraints -- namely, pointwise satisfaction of the HJB equation -- and {inequality} constraints, including positive semidefiniteness and uniform quadratic lower/upper bounds. These features motivate an extension of the nonlinear optimal recovery framework of \cite{chen2021solving}  to accommodate not only equality–type functional constraints but also inequality–type functional constraints. In close analogy with the linear theory, we formulate an abstract setting involving  continuous linear functionals from the dual of the underlying RKHS. The nonlinear variant is also primarily intended for the approximation of solutions to nonlinear PDEs and admits a collocation–type interpretation. Building on the linear methodology, we formulate the problem as follows.

\begin{definition}[Nonlinear optimal recovery in an RKHS]\label{def:NonlinearORP}
Let $\Omega$ be a nonempty set and let $k\colon \Omega \times \Omega \to \R$ be a p.d. kernel with associated RKHS $\mathcal{H}_k(\Omega)$. 
Let $m_p,m_q \in \N_0$ and $n_p,n_q \in \N_0$ satisfy $m_p \ge n_p$ and $m_q \ge n_q$. 
Given maps $p \colon \R^{m_p} \to \R^{n_p}$ and $q \colon \R^{m_q} \to \R^{n_q}$, vectors of continuous linear functionals
\[
\lambda_p \in \bigl(\mathcal{H}_k(\Omega)'\bigr)^{m_p}, 
\qquad
\lambda_q \in \bigl(\mathcal{H}_k(\Omega)'\bigr)^{m_q},
\]
and target vectors $r_p \in \R^{n_p}$ and $r_q \in \R^{n_q}$. The associated nonlinear optimal recovery problem consists in finding
\begin{equation}\label{eq:nonlinear_opt_rec}
 s^* \in \argmin_{s \in \mathcal{H}_k(\Omega)} \left\{  \|s\|_{\mathcal{H}_k(\Omega)}
\;\big|\;
  p\bigl(\lambda_p(s)\bigr) = r_p
  \quad\text{and}\quad
  q\bigl(\lambda_q(s)\bigr) \ge r_q \right\}.
\end{equation}
Here, for a vector of functionals $\lambda = \begin{bmatrix}
    \lambda_1 & \ldots & \lambda_m
\end{bmatrix} \in \bigl(\mathcal{H}_k(\Omega)'\bigr)^m$ (with $m=m_p$ or $m=m_q$) and $s \in \mathcal{H}_k(\Omega)$ we use the componentwise evaluation
\[
\lambda(s) = \begin{bmatrix}
    \lambda_1(s) & \ldots & \lambda_m(s)
\end{bmatrix}^{\top} \in \R^m.
\]
Also the inequality $u \ge v$ for $u,v \in \R^{n_q}$ is understood componentwise.
\end{definition}
\noindent
\noindent At this stage, neither the existence of a {minimizer} of \eqref{eq:nonlinear_opt_rec}  nor a tractable solution strategy for the  infinite-dimensional optimization problem is immediate. Nevertheless, by extending the arguments in \cite{chen2021solving}, we can show the existence of a minimizer and the equivalence to a finite-dimensional optimization problem. 

\begin{theorem}[Finite-dimensional reduction for nonlinear optimal recovery]\label{theo:finiteDimNonlinear}
Given a nonlinear optimal recovery problem according to Definition \ref{def:NonlinearORP} with $p$ and $q$ being additionally continuous. Let
\[
\Lambda \;:=\; \begin{bmatrix}\lambda_p^\top & \lambda_q^\top\end{bmatrix}^\top
\in \bigl(\Hkdual\bigr)^{m_p+m_q}
\]
denote the stacked vector of continuous linear functionals, and assume that the components of $\Lambda$ are linearly independent. 
Assume further that there exists $\bar s\in\Hk$ such that $p(\lambda_p(\bar s))=r_p$ and $q(\lambda_q(\bar s))\ge r_q$. Then
 \eqref{eq:nonlinear_opt_rec} admits a minimizer and 
 it is equivalent to the finite-dimensional problem
\begin{equation}\label{eq:finite_dim_equiv}
  \min_{z_p\in\R^{m_p},\;z_q\in\R^{m_q}} \left \langle
  \begin{bmatrix} z_p \\ z_q \end{bmatrix} , 
  \mathcal{K}_\Lambda^{-1}
  \begin{bmatrix} z_p \\ z_q \end{bmatrix} \right\rangle
  \quad \text{s.t.}\quad
  p(z_p)=r_p,\quad q(z_q)\ge r_q,
\end{equation}
in the following sense:
\begin{itemize}
\item If $(z_p^*,z_q^*)$ solves \eqref{eq:finite_dim_equiv}, then the (unique) minimal-norm interpolant
\[
s^{*}
\;=\;
\argmin_{s\in\Hk}\Bigl\{ \normW{s}{\Hk} \ \Big|\ \lambda_p(s)=z_p^*,\ \lambda_q(s)=z_q^* \Bigr\}
\]
is a minimizer of \eqref{eq:nonlinear_opt_rec}.
\item Conversely, if $s^*$ minimizes \eqref{eq:nonlinear_opt_rec}, then
$z_p^*=\lambda_p(s^*)$ and $z_q^*=\lambda_q(s^*)$ solve \eqref{eq:finite_dim_equiv}.
\end{itemize}
\end{theorem}

\begin{proof}
By linear independence of $\Lambda$, the Gram matrix $\mathcal{K}_\Lambda$ is symmetric positive definite, hence $\mathcal{K}_\Lambda^{-1}$ is symmetric positive definite as well. The feasible function $\bar s$ yields the feasible vector
$\bigl[\lambda_p(\bar s)^\top\ \ \lambda_q(\bar s)^\top\bigr]^\top$ for \eqref{eq:finite_dim_equiv}.
Consider the sublevel set
\[
M_0
\;:=\;
\left\{
 \begin{bmatrix} z_p \\ z_q \end{bmatrix}\in\R^{m_p+m_q}
:\ 
\left \langle \begin{bmatrix} z_p \\ z_q \end{bmatrix} ,
\mathcal{K}_\Lambda^{-1}
\begin{bmatrix} z_p \\ z_q \end{bmatrix} \right \rangle
\le \left\langle
\begin{bmatrix} \lambda_p(\bar s) \\ \lambda_q(\bar s) \end{bmatrix} ,
\mathcal{K}_\Lambda^{-1}
\begin{bmatrix} \lambda_p(\bar s) \\ \lambda_q(\bar s) \end{bmatrix} \right\rangle
\right\}.
\]
Since $\mathcal{K}_\Lambda^{-1}$ is positive definite, $M_0$ is a closed, bounded ellipsoid and hence compact.
By continuity of $p$ and $q$, the sets
\[
M_p:=\left\{\begin{bmatrix} z_p \\ z_q \end{bmatrix}\in\R^{m_p+m_q}: p(z_p)=r_p\right\},\qquad
M_q:=\left\{\begin{bmatrix} z_p \\ z_q \end{bmatrix}\in\R^{m_p+m_q}: q(z_q)\ge r_q\right\}
\]
are closed. Thus, $M:=M_0\cap M_p\cap M_q$ is nonempty by feasibility of $\bar s$ and compact. The quadratic objective in \eqref{eq:finite_dim_equiv} is continuous, hence it attains its minimum on $M$.
\\
Let $(z_p^*,z_q^*)$ be a minimizer of \eqref{eq:finite_dim_equiv}, and let $s^*$ be the associated minimal-norm interpolant defined in the statement. By the standard minimal-norm identity (cf.\ \eqref{eq:min_norm_value}),
\[
\normW{s^*}{\Hk}^2
=
\left \langle \begin{bmatrix} z_p^* \\ z_q^* \end{bmatrix} ,
\mathcal{K}_\Lambda^{-1}
\begin{bmatrix} z_p^* \\ z_q^* \end{bmatrix}\right\rangle.
\]
If there existed $\hat s\in\Hk$ feasible for \eqref{eq:nonlinear_opt_rec} with $\normW{\hat s}{\Hk}<\normW{s^*}{\Hk}$, let $\tilde s$ be the minimal-norm interpolant to the data $\lambda_p(\hat s)$ and $\lambda_q(\hat s)$. Then by the minimal-norm property, it follows
\[
\normW{\tilde s}{\Hk}^2
=
\left\langle \begin{bmatrix} \lambda_p(\hat s) \\ \lambda_q(\hat s) \end{bmatrix},
\mathcal{K}_\Lambda^{-1}
\begin{bmatrix} \lambda_p(\hat s) \\ \lambda_q(\hat s) \end{bmatrix} \right\rangle
\le \normW{\hat s}{\Hk}^2
<
\normW{s^*}{\Hk}^2,
\]
contradicting the optimality of $(z_p^*,z_q^*)$ in \eqref{eq:finite_dim_equiv}.
Hence $s^*$ minimizes \eqref{eq:nonlinear_opt_rec}.\\
Conversely, if $s^*$ minimizes \eqref{eq:nonlinear_opt_rec}, then $(\lambda_p(s^*),\lambda_q(s^*))$ is feasible for \eqref{eq:finite_dim_equiv}. If it were not optimal there, we could repeat the above construction to obtain a feasible function in \eqref{eq:nonlinear_opt_rec} with strictly smaller norm than $s^*$, a contradiction. This establishes both existence and equivalence.
\end{proof}

\noindent
The finite-dimensional program \eqref{eq:finite_dim_equiv} is a strictly convex quadratic optimization problem with (generally nonlinear) constraints and can be handled by standard methods (e.g., Karush-Kuhn-Tucker conditions or interior-point algorithms). In many applications -- most notably for the HJB equation -- the equality constraints can be eliminated explicitly. Suppose that $p$ is {partially affine} in $n_p$ coordinates. This means, after a permutation of coordinates and the functional tuple $\Lambda$, we may write
\begin{equation}\label{eq:resolveStruc}
   p(z_p) \;=\; z_{p,0} \;+\; \tilde p\!\left(z_{\tilde p}\right),
   \qquad
   z_p=\begin{bmatrix} z_{p,0} \\[2pt] z_{\tilde p} \end{bmatrix},
   \quad
   z_{p,0}\in\R^{n_p},\ \ z_{\tilde p}\in\R^{m_p-n_p},
\end{equation}
with $\tilde p\colon \R^{m_p-n_p}\to\R^{n_p}$ continuous. Solving the equality $p(z_p)=r_p$ gives
\begin{equation}\label{eq:solveCon}
   z_{p,0} \;=\; r_p \;-\; \tilde p\!\left(z_{\tilde p}\right).
\end{equation}
With these definitions, the equality constraints in \eqref{eq:finite_dim_equiv} can be resolved explicitly, reducing the problem to
\begin{equation}\label{eq:finite_dim_equiv3}
   \min_{\,z_{\tilde p}\in \R^{m_p-n_p},\; z_q \in \R^{m_q}}
   \left\|
      \begin{bmatrix}
         r_p - \tilde p\!\left(z_{\tilde p}\right) \\[2pt]
         z_{\tilde p} \\[2pt]
         z_q
      \end{bmatrix}
   \right\|_{{\mathcal{K}}_{\Lambda}^{-1}}^{2}
   \quad \text{s.t.} \quad
   q(z_q) \;\ge\; r_q.
\end{equation}
The reduced program \eqref{eq:finite_dim_equiv3} is a nonlinear
{weighted least–squares} problem with inequality constraints. Extending
the approach of \cite{chen2021solving}, we adopt a Gauss-Newton scheme that
linearizes only the objective mapping while {retaining} the inequality
constraints: 
Define the residual mapping
\[
   \rho(z_{\tilde p},z_q)
   \;:=\;
   \begin{bmatrix}
      r_p - \tilde p(z_{\tilde p}) \\[2pt]
      z_{\tilde p} \\[2pt]
      z_q
   \end{bmatrix}
   \in \R^{m_p+m_q}.
\]
Given an iterate $\bigl(z_{\tilde p}^{\mathrm{old}},z_q^{\mathrm{old}}\bigr)$, linearize only the residual $\rho$ while retaining the original inequality constraints:
\begin{equation}\label{eq:finite_dim_equiv4}
   \min_{\Delta z_{\tilde p}\in\R^{m_p-n_p},\;\Delta z_q\in\R^{m_q}}
   \left\|
      \rho\!\left(z_{\tilde p}^{\mathrm{old}},z_q^{\mathrm{old}}\right)
      +
      J_{\rho}\!\left(z_{\tilde p}^{\mathrm{old}},z_q^{\mathrm{old}}\right)
      \begin{bmatrix}\Delta z_{\tilde p}\\[2pt]\Delta z_q\end{bmatrix}
   \right\|_{{\mathcal{K}}_{\Lambda}^{-1}}^{2}
   \quad\text{s.t.}\quad
   q\!\left(z_q^{\mathrm{old}}+\Delta z_q\right)\;\ge\; r_q,
\end{equation}
where the Jacobian of $\tilde p$ is $J_{\tilde p}(z_{\tilde p})\in\R^{n_p\times (m_p-n_p)}$ and
\[
   J_{\rho}\left(z_{\tilde p}^{\mathrm{old}},z_q^{\mathrm{old}}\right)
   \;=\;
   \begin{bmatrix}
      -\,J_{\tilde p}(z_{\tilde p}^{\mathrm{old}}) & 0 \\[2pt]
      I_{m_p-n_p} & 0 \\[2pt]
      0 & I_{m_q}
   \end{bmatrix}
\]
is the Jacobian of $\rho$.
With a computed step $(\Delta z_{\tilde p},\Delta z_q)$ solving \eqref{eq:finite_dim_equiv4}, the update reads
\[
   \begin{bmatrix}
      z_{\tilde p}^{\mathrm{new}} \\[2pt] z_q^{\mathrm{new}}
   \end{bmatrix}
   =
   \begin{bmatrix}
      z_{\tilde p}^{\mathrm{old}} \\[2pt] z_q^{\mathrm{old}}
   \end{bmatrix}
   +
   \begin{bmatrix}
      \Delta z_{\tilde p} \\[2pt] \Delta z_q
   \end{bmatrix}.
\]
The quadratic subproblem \eqref{eq:finite_dim_equiv4} is in fact a {linear} optimal recovery problem (with inequality constraints), obtained by retracing the constructions from \eqref{eq:nonlinear_opt_rec} to \eqref{eq:finite_dim_equiv4}. This is formalized next.
\begin{theorem}[Gauss-Newton step as a linear optimal recovery problem]\label{lem:GaussNewton-vector}
Assume the hypotheses of Theorem~\ref{theo:finiteDimNonlinear}. 
Suppose $p$ admits the partially affine decomposition in \eqref{eq:resolveStruc} with  $\tilde p\in C^{1}\!\bigl(\R^{m_p-n_p},\R^{n_p}\bigr)$. 
Write the vector of functionals accordingly as
\[
  \lambda_p \;=\; \begin{bmatrix} \lambda_{p,0} \\ \lambda_{\tilde p}\end{bmatrix},
  \qquad
  \lambda_{p,0}\in(\Hkdual)^{n_p},\ \ \lambda_{\tilde p}\in(\Hkdual)^{m_p-n_p}.
\]
Given a current RKHS iterate $s^{\mathrm{old}}\in\Hk$, set
\begin{equation}\label{eq:oldz-def-vector}
   z_{\tilde p }^{\mathrm{old}} \;:=\; \lambda_{\tilde p}\bigl(s^{\mathrm{old}}\bigr)\in\R^{m_p-n_p}.
\end{equation}
Define the {linearized equality functionals} and right-hand side
\begin{equation}\label{eq:linfun-vector}
  \lambda^{\mathrm{lin}}_{s^{\mathrm{old}}}
  \;:=\;
  \lambda_{p,0}\;+\; J_{\tilde p}\bigl( z_{\tilde p }^{\mathrm{old}}\bigr)\,\lambda_{\tilde p}
  \ \in (\Hkdual)^{n_p},
  \qquad
  b^{\mathrm{old}}
  \;:=\;
  r_p - \tilde p \bigl( z_{\tilde p }^{\mathrm{old}}\bigr)
        + J_{\tilde p} \bigl( z_{\tilde p }^{\mathrm{old}}\bigr)\,  z_{\tilde p }^{\mathrm{old}}
  \ \in \R^{n_p}.
\end{equation}
Then the Gauss-Newton subproblem \eqref{eq:finite_dim_equiv4} is equivalent to the infinite-dimensional linear optimal recovery problem
\begin{equation}\label{eq:sNewGauss-vector}
 s_{\text{GN}}^*\in \argmin_{s\in\Hk} \left\{ \|s\|_{\Hk}\;\big|\;
  \lambda^{\mathrm{lin}}_{s^{\mathrm{old}}}(s)\;=\;b^{\mathrm{old}}\quad\text{and}\quad
  q\bigl(\lambda_q(s)\bigr)\ \ge\ r_q \right\}
 .
\end{equation}
Moreover, the components of $\lambda^{\mathrm{lin}}_{s^{\mathrm{old}}}$ and $\lambda_q$ are linearly independent. 
\end{theorem}
\begin{proof}
By a linear change of variables, \eqref{eq:finite_dim_equiv4} is equivalent to minimization in the new variables (omitting the superscript ``new''):
\begin{equation}\label{eq:GN-lin-vec}
  \min_{ z_{\tilde p}\in\R^{m_p-n_p},\; z_q\in\R^{m_q}}\ 
  \bigl\|\rho_{\mathrm{lin}}(z_{\tilde p},z_q)\bigr\|_{{\mathcal{K}}_{\Lambda}^{-1}}^{2}
  \quad\text{s.t.}\quad
  q(z_q)\ge r_q,
\end{equation}
with the {linearized residual}
\[
  \rho_{\mathrm{lin}}(z_{\tilde p},z_q)
  =
  \rho(z_{\tilde p}^{\mathrm{old}},z_q^{\mathrm{old}})
    +  J_{\rho}\left(z_{\tilde p}^{\mathrm{old}},z_q^{\mathrm{old}}\right)\left(
    \begin{bmatrix}z_{\tilde p}\\ z_q\end{bmatrix}
    - 
    \begin{bmatrix}z_{\tilde p}^{\mathrm{old}}\\ z_q^{\mathrm{old}}\end{bmatrix} \right)
  =
  \begin{bmatrix}
    r_p - \tilde p(z_{\tilde p}^{\mathrm{old}})
      + J_{\tilde p}(z_{\tilde p}^{\mathrm{old}})\bigl(z_{\tilde p}^{\mathrm{old}}-z_{\tilde p}\bigr)\\[2pt]
    z_{\tilde p}\\[2pt]
    z_q
  \end{bmatrix}.
\]
Now {reverse} the explicit elimination in \eqref{eq:solveCon},  define
\begin{equation*}
  z_{p,0}
  :=
  r_p - \tilde p\bigl(z_{\tilde p}^{\mathrm{old}}\bigr)
  + J_{\tilde p}\bigl(z_{\tilde p}^{\mathrm{old}}\bigr)\bigl(z_{\tilde p}^{\mathrm{old}}-z_{\tilde p}\bigr)
  \in\R^{n_p}
\end{equation*}
and note that this is equivalent to the {linear} constraint
\begin{equation*}
  z_{p,0}
  + J_{\tilde p}\bigl(z_{\tilde p}^{\mathrm{old}}\bigr)\,z_{\tilde p}
  \;=\;
  r_p - \tilde p\bigl(z_{\tilde p}^{\mathrm{old}}\bigr)
  + J_{\tilde p}\bigl(z_{\tilde p}^{\mathrm{old}}\bigr)\,z_{\tilde p}^{\mathrm{old}}
  \;=\; b^{\mathrm{old}}.
\end{equation*}
Thus, introducing this linear constraint and the new variable in  \eqref{eq:GN-lin-vec} yields the equivalent  finite-dimensional quadratic program
\begin{equation*}
   \min_{ z_{p,0}\in\R^{n_p},\; z_{\tilde p}\in\R^{m_p-n_p},\; z_q\in\R^{m_q}}\ 
  \left\| \begin{bmatrix}
      z_{p,0} \\
      z_{\tilde p} \\ 
      z_q
  \end{bmatrix}\right\|_{{\mathcal{K}}_{\Lambda}^{-1}}^{2}
  \quad\text{s.t.}\quad
  z_{p,0} + J_{\tilde p}(z_{\tilde p}^{\mathrm{old}})\,z_{\tilde p} = b^{\mathrm{old}},
  \ \ q(z_q)\ge r_q .
\end{equation*}
Finally, apply Theorem~\ref{theo:finiteDimNonlinear} and identifying
$
  \lambda^{\mathrm{lin}}_{s^{\mathrm{old}}},
$
gives the equivalence  of \eqref{eq:finite_dim_equiv4} to \eqref{eq:sNewGauss-vector}. \\
Linear independence of the components of $\lambda^{\mathrm{lin}}_{s^{\mathrm{old}}}$ and $\lambda_q$ follows because
\[
  \begin{bmatrix} \lambda^{\mathrm{lin}}_{s^{\mathrm{old}}}\\ \lambda_q \end{bmatrix}
  =
  \begin{bmatrix}
    I_{n_p} & J_{\tilde p}(z_{\tilde p}^{\mathrm{old}}) & 0\\
    0 & 0 & I_{m_q}
  \end{bmatrix}
  \begin{bmatrix} \lambda_{p,0}\\ \lambda_{\tilde p}\\ \lambda_q \end{bmatrix},
\]
and the block matrix on the left has full row rank while
$\{\lambda_{p,0},\lambda_{\tilde p},\lambda_q\}$ is linearly independent by hypothesis.
\end{proof}

\noindent
Recasting the Gauss-Newton subproblem \eqref{eq:finite_dim_equiv4} as the
{infinite–dimensional} linear optimal recovery problem
\eqref{eq:sNewGauss-vector} yields a concrete computational benefit.
The number of linear equality functionals drops from
$m_p$ to $n_p$ (while retaining the $m_q$ inequality constraints), since the original $m_p$ equality functionals are
aggregated into the $n_p$ linearized functionals in
$\lambda^{\mathrm{lin}}_{s^{\mathrm{old}}}$.

\medskip
\noindent
Moreover, by solving the  linear optimal recovery problem
without the inequality constraints
\begin{equation}\label{eq:sNewGauss2}
  s^{\mathrm{new}}
  =
  \operatorname*{arg\,min}_{s \in \Hk} \left\{ \normW{s}{\Hk}
  \;\big|\;
  \lambda^{\mathrm{lin}}_{s^{\mathrm{old}}}(s) \;=\; b^{\mathrm{old}}\right\},
\end{equation}
one also solves the constrained Gauss-Newton subproblem
\eqref{eq:sNewGauss-vector} provided the inequality constraints are satisfied
at $s^{\mathrm{new}}$, i.e.,
\begin{equation}\label{eq:sNewGauss3}
  q \bigl(\lambda_q(s^{\mathrm{new}})\bigr) \;\ge\; r_q.
\end{equation}
Indeed, any feasible point of \eqref{eq:sNewGauss-vector} satisfies the same
linearized equalities as in \eqref{eq:sNewGauss2}; since
$s^{\mathrm{new}}$ has minimal $\Hk$-norm among {all} functions
satisfying these equalities, it is also minimal over the subset that additionally
obeys the inequalities, hence solves \eqref{eq:sNewGauss-vector} whenever
\eqref{eq:sNewGauss3} holds.

\medskip
\noindent
From a computational viewpoint, \eqref{eq:sNewGauss2} requires solving only an
$n_p\times n_p$ linear system (the Gram matrix associated with the functionals in $\lambda^{\mathrm{lin}}_{s^{\mathrm{old}}}$), while checking its validity by \eqref{eq:sNewGauss3}.
This strategy is particularly effective when the minimizer has no active inequality constraints and the iteration is initialized sufficiently close to that minimizer. In this regime, the classical local convergence theory for Gauss-Newton applies to the procedure that iteratively computes \eqref{eq:oldz-def-vector} and solves \eqref{eq:sNewGauss2}; see \cite{deuflhard1979affine,chen2005convergence,ferreira2011local} for precise assumptions and rates.\\
\noindent Theorem~\ref{lem:GaussNewton-vector} is the reason for the fact that the RKHS-based PI
scheme of the previous subsection is {algorithmically equivalent} to the
practical solution ansatz for the verification–based approximation of the
OVF. The latter is introduced next, building on the
nonlinear optimal recovery framework developed in the current subsection.

\section{Optimal Recovery of the Optimal Value Function}
\label{sec:opt_rec_vf}
In this section, we instantiate the nonlinear optimal recovery framework developed above for the verification conditions stated in Corollaries~\ref{coro:VerificAnalytic} and~\ref{coro:VerificConstant}. This leads in our case to two distinct optimal recovery formulations: First, the OVF $v^{*}$ is real-analytic on $\Omega$, second, $v^{*}$ satisfies uniform quadratic lower and upper bounds on the domain. Throughout the presentation, we adopt the standing assumption that the unknown OVF $v^{*}$ belongs to the RKHS $\mathcal{H}_k(\Omega)$ used for the approximation.

\subsection{Optimal Recovery of an Analytic OVF}

In this subsection we assume that the OVF $v^{*}$ is
real–analytic on $\Omega$, so that Corollary~\ref{coro:VerificAnalytic}
applies. In particular, the admissible approximants must also be analytic.
We therefore choose an RKHS consisting of analytic functions. A canonical
example is the Gaussian kernel, whose associated
RKHS is composed of real–analytic functions; see \cite{steinwart2008support}.

\medskip

\noindent
Given a p.d. kernel $k$ with analytic RKHS and a finite set of pairwise
distinct points $X_{n}=\{x_{1},\ldots,x_{n}\}$,
we pose the following nonlinear optimal recovery problem 
\begin{equation}\label{eq:minProblem}
\min_{v \in \Hk}   \|v\|_{\Hk}
 \text{ s.t. } \begin{cases}  
 \langle f(x_i), \nabla v(x_i) \rangle
  - \frac{1}{4} \normW{ g(x_i)^{\top}\, \nabla v(x_i)}{R^{-1}}^2  + h(x_i) = 0,  & i=1,\ldots,n,\\
 v(x_i) \ge 0, & i=1,\ldots,n,\\
 v(0) = 0.
\end{cases}
\end{equation}
\noindent
To cast \eqref{eq:minProblem} in the abstract vectorial framework from
Section~\ref{sec:nonlinearOp}, we first recall that by
Lemma~\ref{lem:point-and-grad-eval} both point evaluation $\delta_x\in\Hkdual$
and directional gradient evaluation
$\langle a,\delta_x   \circ\nabla(\,\cdot\,)\rangle\in\Hkdual$
belong to the dual space for every $a\in\R^{N}$ and $x\in\Omega$. Next, we define
\begin{align}
\lambda_{p,0}(\,\cdot\,) &:= \begin{bmatrix}
         \delta_0(\,\cdot\,) & \langle f(x_1), \delta_{x_1}   \circ \nabla (\,\cdot\,) \rangle & 
         \cdots & \langle f(x_n), \delta_{x_n}   \circ \nabla (\,\cdot\,) \rangle
    \end{bmatrix}^{\top}, \notag  \\
    \lambda_{\tilde p}(\,\cdot\,) &:= \big[ 
           \langle g_1(x_1), \delta_{x_1}  \circ \nabla (\,\cdot\,) \rangle \;\;\; \cdots \;\;\;  \langle g_M(x_1), \delta_{x_1}   \circ \nabla (\,\cdot\,) \rangle \notag\\
           & \quad \;\;\; \cdots \;\;\;  \langle g_1(x_n), \delta_{x_n}   \circ \nabla (\,\cdot\,) \rangle \;\;\; \cdots \;\;\;  \langle g_M(x_n), \delta_{x_n}   \circ \nabla (\,\cdot\,) \rangle
    \big]^{\top}, \notag \\
    \lambda_p(\,\cdot\,) &:= \begin{bmatrix}
        \left(\lambda_{p,0}(\,\cdot\,)\right)^{\top} &  \left(\lambda_{\tilde p}(\,\cdot\,)\right)^{\top}
    \end{bmatrix}^{\top} \in \left( \Hkdual \right)^{1+n+nM}, \label{eq:lamPCon} \\
    \lambda_q(\,\cdot\,) &:= \begin{bmatrix}
        \delta_{x_1}(\,\cdot\,) &   \cdots &   \delta_{x_n}(\,\cdot\,)
    \end{bmatrix}^{\top} \in \left( \Hkdual \right)^{n}, \label{eq:lamQCon}
\end{align}
where we denote by $g_1(x_i),...,g_M(x_i)$ the columns of $g(x_i)$.   Furthermore, we set 
\begin{align}
    p(\lambda_p(s)) &:= \begin{bmatrix}
         s(0) & \langle f(x_1),   \nabla s(x_1) \rangle & 
         \cdots & \langle f(x_n),  \nabla s(x_n) \rangle
    \end{bmatrix}^{\top} + \tilde p (\lambda_{\tilde p}(s)), \label{eq:pCon}\\
    \tilde p (\lambda_{\tilde p}(s)) &:= \begin{bmatrix}
         0 & - \frac{1}{4} \normW{ g(x_1)^{\top}\, \nabla s(x_1)}{R^{-1}}^2 & 
         \cdots & - \frac{1}{4} \normW{ g(x_n)^{\top}\, \nabla s(x_n)}{R^{-1}}^2
    \end{bmatrix}^{\top}, \label{eq:tildepCon}\\
    r_p &:= \begin{bmatrix}
         0 & -h(x_1) & 
         \cdots & -h(x_n)
    \end{bmatrix}^{\top},\label{eq:tildepCon2}\\
    q(\lambda_q(s)) &:= \begin{bmatrix}
         s(x_1) & 
         \cdots & s(x_n)
    \end{bmatrix}^{\top}, \notag\\
        r_q &:= \begin{bmatrix}
         0 &   
         \cdots & 0
    \end{bmatrix}^{\top}. \notag
\end{align}
 With this notation, \eqref{eq:minProblem}
can be written in the abstract form
\begin{gather}\label{eq:minProblem0}
\min_{v \in \Hk}    \|v\|_{\Hk} 
 \text{ s.t. } p(\lambda_p(v)) = r_p \quad\text{and}\quad q(\lambda_q(v)) \geq  r_q.
\end{gather}
\noindent For a given center set $X_n$, denote by \[
    \mathcal{M}_{X_{n}} \;:=\;
    \Bigl\{v\in \Hk\;:\; v \text{ solves \eqref{eq:minProblem0} for $X_n$}\Bigr\}
\subset \Hk\] the set of minimizers of the optimal recovery problem~\eqref{eq:minProblem}.
Generally, a main motivation for an optimal recovery ansatz is that it yields a straightforward convergence proof when the set of centers is enlarged so that their union becomes dense in  $\Omega$.
\begin{theorem}\label{tho:OptRecovAna}
Assume that Assumption~\ref{as:data} holds, let $\Omega\subset\mathbb{R}^{N}$ be a domain containing the origin, and let $(X_n)_{n\in\mathbb{N}}\subset \Omega \setminus \{0 \}$ be a nested sequence of finite subsets whose union is dense in $\Omega$, i.e.,
\[
\cl{\bigcup_{n\in\mathbb{N}} X_n} = \Omega .
\]
Suppose the RKHS associated with the kernel $k$ consists of analytic functions,
i.e., $\mathcal{H}_k(\Omega)\subset A(\Omega)$, and that the OVF satisfies $v^{*}\in \mathcal{H}_k(\Omega)$. Then any sequence $(v_n)_{n\in\mathbb{N}}$ with $v_n \in \mathcal{M}_{X_n}$ converges to $v^*$ in the $\Hk$-norm:
\[
\lim_{n\to\infty} \|v_n - v^*\|_{\Hk} = 0 .
\]
\end{theorem}

\begin{proof}
The argument follows \cite{Bacchetta2025} with minor modifications.\\
\textbf{Step 1 (uniform boundedness).}
Because $v^{*}$ is admissible for every discrete problem, minimality yields
\(
\|v_n\|_{\Hk}\le\|v^{*}\|_{\Hk}
\)
for all $n \in \mathbb{N}$.  Hence $(v_n)_{n \in \mathbb{N}}$ is contained in the closed ball of radius
$\|v^{\ast}\|_{\Hk}$; by Banach–Alaoglu this ball is weakly compact in the Hilbert space $\Hk$ and therefore $(v_n)_{n \in \mathbb{N}}$ possesses weak accumulation points.\\
\noindent \textbf{Step 2 (identification of every weak limit).}
Let $\tilde v$ be an arbitrary weak limit of a subsequence, i.e., $v_{n_\ell}\rightharpoonup\tilde v$.  
For each $x \in \Omega$, $\delta_{x}(\,\cdot\,)   $ and $\delta_{x}(\,\cdot\,)  \circ\partial^{s}(\,\cdot\,)$ for $s=1,...,N$   are continuous linear functionals on $\Hk$; thus
\[
v_{n_\ell}(x)\to\tilde v(x), 
\quad 
\nabla v_{n_\ell}(x)\to\nabla\tilde v(x)\;\; \text{ for } \;\; \ell \rightarrow \infty.
\]
Fix $x\in\bigcup_{n\in \mathbb{N}}X_n$ and choose $n_0$ with $x\in X_{n_0}$.  
The discrete constraints imposed on $v_{n_\ell}$ give
\[
\langle f(x),\nabla  v_{n_\ell}(x)\rangle- \frac{1}{4}\normW{g(x)^{\top}\nabla  v_{n_\ell}(x)}{R^{-1}}^2+h(x)=0,\quad
  v_{n_\ell}(x)\ge0,\quad
  v_{n_\ell}(0)=0.
\]
for $n_\ell>n_0$. Thus, passing to the limit and using the continuity of the norm and scalar product, we obtain
\[
\langle f(x),\nabla\tilde v(x)\rangle- \frac{1}{4}\normW{g(x)^{\top}\nabla\tilde v(x)}{R^{-1}}^2
+h(x)=0,\quad
\tilde v(x)\ge0,\quad
\tilde v(0)=0.
\]
By density of $\bigcup_{n\in \mathbb{N}}X_n$ and continuity of the terms involved, the same hold for every $x\in\Omega$.  
Since $\tilde v\in \Hk \subset A(\Omega)$, Corollary~\ref{coro:VerificAnalytic} implies $\tilde v=v^{*}$ on~$\Omega$.  
Thus {every} weak limit of $(v_n)_{n \in \mathbb{N}}$ coincides with $v^{*}$; consequently
\(v_n\rightharpoonup v^{*}\) in $\Hk$.\\
\noindent \textbf{Step 3 (upgrade to strong convergence).}
Since $v_{n+1}$ is also in the feasibility set for the minimal-norm problem corresponding to $X_n$, we have
\begin{align*}
     \Vert v_n \Vert_{\Hk} \leq \Vert v_{n+1} \Vert_{\Hk} \leq \Vert v^* \Vert_{\Hk}  \quad\text{for all } n \in \mathbb{N},
\end{align*}
so   $\left(\Vert v_n \Vert_{\Hk}\right)_{n \in \mathbb{N}}$ is a monotone increasing and bounded sequence, and therefore convergent. Furthermore, weak convergence implies
$$\langle v^* ,w \rangle_{\Hk} = \lim_{n \rightarrow \infty} \langle v_n ,w \rangle_{\Hk} \leq \Vert  w \Vert_{\Hk} \lim_{n \rightarrow \infty} \Vert  v_n \Vert_{\Hk}, \qquad \forall w \in \Hk$$
which, together with the norm bound, gives
\[
\|v^{\ast}\|_{\Hk} = \sup_{\Vert  w \Vert=1} | \langle v^* ,w \rangle_{\Hk}| 
\le\lim_{n\to\infty}\|v_n\|_{\Hk}
\le\|v^{\ast}\|_{\Hk},
\]
so $\|v_n\|_{\Hk}\to\|v^{\ast}\|_{\Hk}$.
Finally,
\[
\|v_n-v^{\ast}\|_{\Hk}^{2}
=\|v_n\|_{\Hk}^{2}-2\langle v_n,v^{\ast}\rangle_{\Hk}+\|v^{\ast}\|_{\Hk}^{2}\xrightarrow[n\to\infty]{}0,
\]
because the middle term converges by weak convergence, and the first terms converge as shown before.  Hence $v_n\to v^{\ast}$ strongly in $\Hk$.

\end{proof}
\subsection{Optimal recovery of an OVF under quadratic bounds}
\label{subsec:opt-recovery-quadratic}
\noindent Theorem~\ref{tho:OptRecovAna} ensures convergence under the strong assumption
$v^{*}\in\mathcal{H}_{k}(\Omega)\subset A(\Omega)$ with dense sampling. Since
analyticity is often unclear, we now adopt the two-sided quadratic bounds
of Corollary~\ref{coro:VerificConstant}, but obtain only {local} convergence. So we consider an OVF \(v^{*}\) that satisfies 
\[
  \alpha^{*}\,\|x\|^{2}\ \le\ v^{*}(x)\ \le\ \beta^{*}\,\|x\|^{2},
  \qquad x\in\Omega,
\]
which is Assumption \ref{as:data2}. Throughout, fix constants \(0<\alpha\le \alpha^{*}\) and \(\beta^{*}\le \beta\). For pairwise distinct sampling sites
\(X_{n}:=\{x_{1},\dots,x_{n}\}\subset \Omega\setminus\{0\}\), and a p.d. kernel \(k\) with RKHS \(\mathcal{H}_{k}(\Omega)\), we pose
the following nonlinear optimal recovery problem:
\begin{equation}
\label{eq:minProblem2}
\min_{v \in \Hk}   \|v\|_{\Hk}
 \text{ s.t. } \begin{cases}  
 \langle f(x_i), \nabla v(x_i) \rangle
  - \frac{1}{4}\normW{g(x_i)^{\top} \nabla v(x_i)}{R^{-1}}^2 + h(x_i) = 0,  & i=1,\ldots,n,\\
\beta \Vert x_i \Vert^2 \ge v(x_i) \ge \alpha \Vert x_i \Vert^2, & i=1,\ldots,n,\\
 v(0) = 0.
\end{cases}
\end{equation}
Again, also problem \eqref{eq:minProblem2} can be expressed abstractly as
\begin{gather}\label{eq:minProblem20}
\min_{v \in \Hk}    \|v\|_{\Hk} 
 \text{ s.t. } p(\lambda_p(v)) = r_p \quad\text{and}\quad r_{q,u} \geq q(\lambda_q(v)) \geq  r_{q,l}
\end{gather}
using the definitions \eqref{eq:pCon}--\eqref{eq:tildepCon2} and  
\begin{align*}
    r_{q,u} &:= \begin{bmatrix}
         \beta \Vert x_1 \Vert^2 &  
         \cdots & \beta \Vert x_n \Vert^2
    \end{bmatrix}^{\top}\\
  r_{q,l}  &:= \begin{bmatrix}
         \alpha \Vert x_1 \Vert^2 &  
         \cdots & \alpha \Vert x_n \Vert^2
    \end{bmatrix}^{\top}.
\end{align*}
The corresponding set of admissible minimizers is
\[
    \mathcal{M}_{X_{n},\alpha,\beta} \;:=\;
    \Bigl\{v\in \Hk\;:\; v \text{ solves \eqref{eq:minProblem20} for $X_n$}\Bigr\}.
\]
As in the analytic case, convergence can be established, though only on a suitable subdomain of~\(\Omega\).
\begin{theorem}\label{thm:convergence-quadratic}
Let \(\Omega\subset\mathbb{R}^{N}\) be a bounded domain containing the origin, let Assumption \ref{as:data} hold and
 fix numbers \(0<\alpha\le \alpha^{*}\) and \(\beta^{*}\le \beta\).
Let \(\bigl(X_{n}\bigr)_{n\in\mathbb{N}}\subset\Omega\setminus\{0\}\) be a nested family of finite sets with dense union,
\(\cl{\bigcup_{n\in\mathbb{N}}X_{n}}=\Omega\).
Suppose, moreover, that for the p.d. kernel \(k\) under consideration the OVF \(v^{*}\) belongs to \(\Hk\).
Then there exists a subdomain \(\tilde{\Omega}\subset\Omega\) containing the origin such that, for any sequence \(\bigl(v_{n}\bigr)_{n\in\mathbb{N}}\) with \(v_{n}\in \mathcal{M}_{X_{n},\alpha,\beta}\),
\[
  \lim_{n\to\infty}\,\|v_{n}-v^{*}\|_{\mathcal{H}_{k}(\tilde{\Omega})}\;=\;0.
\]
\end{theorem}

\begin{proof}
By Corollary~\ref{coro:VerificConstant} there exists a subdomain
\(\tilde{\Omega}\subset\Omega\), depending only on \(\Omega\), \(\beta\), \(\alpha\) and $v^*$,
on which the OVF is the unique solution of the constraints
(1)--(2) in Theorem~\ref{thm:VerOpt} together with condition~(3) of
Corollary~\ref{coro:VerificConstant}.
Further, note that for any \(f\in \Hk\) we have the restriction estimate
\(\|f|_{\tilde{\Omega}}\|_{H_{k}(\tilde{\Omega})}
   \le\|f\|_{\Hk}\)
(see \cite[Theorem~10.47]{wendland2004}).  
Since \(v^{*}\) satisfies all constraints with constants \(\alpha\) and \(\beta\),
we infer for every \(n\in\mathbb{N}\)
\[
    \|v_{n}|_{\tilde{\Omega}}\|_{H_{k}(\tilde{\Omega})}
    \;\le\;\|v_{n}\|_{\Hk}
    \;\le\;\|v^{*}\|_{\Hk} .
\]
Hence \(\bigl(v_{n}|_{\tilde{\Omega}}\bigr)_{n\in\mathbb{N}}\) is bounded in
\(H_{k}(\tilde{\Omega})\) and possesses weak accumulation points, 
all of which coincide with \(v^{*}\) on $\tilde{\Omega}$ by uniqueness as in Step~2 of the proof of Theorem~\ref{tho:OptRecovAna}.
Weak convergence therefore holds, and strong convergence follows exactly as in Step~3 of the proof of Theorem~\ref{tho:OptRecovAna}.
\end{proof}
\noindent So far, we have proved convergence for the abstract infinite-dimensional minimization problems \eqref{eq:minProblem} and \eqref{eq:minProblem2}. To obtain a practical solver, we now invoke the finite–dimensional reduction of Theorem~\ref{theo:finiteDimNonlinear} and subsequently employ an iterative Gauss–Newton method, thereby returning to the RKHS–PI formulation.

\subsection{Policy Iteration via Gauss-Newton Linearization}

For the numerical solution of \eqref{eq:minProblem} and \eqref{eq:minProblem2}, the Gauss-Newton scheme from Subsection~\ref{sec:nonlinearOp} provides an effective strategy. To reduce \eqref{eq:minProblem} and \eqref{eq:minProblem2} to finite-dimensional problems, we require the hypotheses of Theorem~\ref{theo:finiteDimNonlinear}. The continuity of the functions $p$ and $q$ is given, and feasibility is ensured, because the OVF  satisfies all the constraints. The only remaining issue is the linear independence of the functionals collected in $\Lambda$. For this, we impose the following rank condition.

\begin{assumption}\label{ass:nonzero-centers}
 For $N \geq M+1$, let $X_n=\{x_i\}_{i=1}^n \subset \Omega\setminus\{0\}$ be pairwise distinct points such that
  \[
    \operatorname{rank}\!\begin{bmatrix}
      f(x_i) & g_1(x_i) & \cdots & g_M(x_i)
    \end{bmatrix} = M+1,
    \qquad \text{for all } i=1,\dots,n.
  \]
\end{assumption}
\noindent A slightly weaker requirement would be that the $f$-column is not contained in the span of the $g$-columns, i.e.,
\[
  f(x_i) \notin \operatorname{span}\!\bigl\{ g_1(x_i),\dots,g_M(x_i) \bigr\}.
\]
Proceeding under this weaker hypothesis, however, would lead to considerably heavier notation, so we retain Assumption~\ref{ass:nonzero-centers}. In all numerical examples below, we employ a nested sequence of finite sets $\bigl(X_n\bigr)_{n\in\mathbb{N}}\subset \Omega \setminus \{0 \}$ with dense union,
\[
  \cl{\bigcup_{n\in\mathbb{N}} X_n}=\Omega,
\]
such that each $X_n$ satisfies Assumption~\ref{ass:nonzero-centers}. Moreover, by Lemma~\ref{theo:KernelLinFunc}, the assumptions stated there together with Assumption~\ref{ass:nonzero-centers} imply that all linear functionals appearing in $
\Lambda \;:=\; \begin{bmatrix}\lambda_p^\top & \lambda_q^\top\end{bmatrix}^\top
$
with $\lambda_p$ from \eqref{eq:lamPCon} and $\lambda_q$ from \eqref{eq:lamQCon} 
are linearly independent. Furthermore, the equality constraints in \eqref{eq:minProblem0} and \eqref{eq:minProblem20} are partially affine with the  nonlinear remainder term $\tilde p (\lambda_{\tilde p}(s))$ from \eqref{eq:tildepCon}. This structure is required both for their resolution and for the applicability of the Gauss–Newton scheme outlined in Subsection~\ref{sec:nonlinearOp}. Consequently, each Gauss-Newton step is equivalent to an infinite-dimensional {linear} optimal recovery problem with inequality constraints. To identify the linear functionals in \eqref{eq:linfun-vector} and the associated right-hand sides, fix an iterate $v^{\mathrm{old}} \in \Hk$ and set
\begin{align*}
  z_{\tilde p }^{\mathrm{old}} &=  \lambda_{\tilde p}\bigl(v^{\mathrm{old}}\bigr) \\
  &= \big[
           \langle g_1(x_1),   \nabla v^{\mathrm{old}}(x_1) \rangle \;\;\; \cdots \;\;\; \langle g_M(x_1),   \nabla v^{\mathrm{old}}(x_1) \rangle \\
           & \quad \;\;\; \cdots \;\;\;  \langle g_1(x_n),   \nabla v^{\mathrm{old}}(x_n) \rangle \;\;\; \cdots \;\;\;  \langle g_M(x_n),   \nabla v^{\mathrm{old}}(x_n) \rangle
    \big]^{\top}.
\end{align*}
The vector of linear functionals is
\begin{align*}
  \lambda^{\mathrm{lin}}_{v^{\mathrm{old}}}(\,\cdot\,) 
  = & \,
  \lambda_{p,0}(\,\cdot\,)\;+\; J_{\tilde p}\bigl( z_{\tilde p }^{\mathrm{old}}\bigr)\,\lambda_{\tilde p}(\,\cdot\,) \\
  =&\bigg[
         \delta_0(\,\cdot\,) \quad \langle f(x_1), \delta_{x_1}  \circ \nabla (\,\cdot\,) \rangle  - \frac{1}{2} \left\langle g(x_1)^{\top} \nabla v^{\mathrm{old}}(x_1), g(x_1)^{\top} \delta_{x_1}   \circ \nabla (\,\cdot\,) \right \rangle_{R^{-1}} \\ &   \quad 
         \cdots \quad \langle f(x_n), \delta_{x_n}   \circ \nabla (\,\cdot\,) \rangle- \frac{1}{2} \left\langle g(x_n)^{\top} \nabla v^{\mathrm{old}}(x_n), g(x_n)^{\top} \delta_{x_n}  \circ \nabla (\,\cdot\,) \right \rangle_{R^{-1}}
    \bigg]^{\top}
\end{align*}
and the corresponding  right-hand side is
\begin{align*}
     b^{\mathrm{old}}
 & =
  r_p - \tilde p \bigl( z_{\tilde p }^{\mathrm{old}}\bigr)
        + J_{\tilde p} \bigl( z_{\tilde p }^{\mathrm{old}}\bigr)\,  z_{\tilde p }^{\mathrm{old}} \\
        & = \begin{bmatrix}
         0 & -h(x_1) - \frac{1}{4} \normW{g(x_1)^{\top}\nabla v^{\mathrm{old}}(x_1)}{R^{-1}}^2 & 
         \cdots & -h(x_n) - \frac{1}{4} \normW{g(x_n)^{\top}\nabla v^{\mathrm{old}}(x_n)}{R^{-1}}^2
    \end{bmatrix}.
\end{align*} Consequently, with the control defined pointwise by
\[
  u^{\mathrm{old}}(x_i)
  \;:=\; -\tfrac{1}{2}\,R^{-1}\,g(x_i)^{\!\top}\,\nabla v^{\mathrm{old}}(x_i) \quad\text{for }i=1,...,n
\]
the equality constraints in $\lambda^{\mathrm{lin}}_{v^{\mathrm{old}}}(v)\;=\;b^{\mathrm{old}}$ reduce to $v(0) = 0$ and  the standard policy-evaluation relation
\begin{equation*}
  \bigl\langle f(x_i) + g(x_i)\,u^{\mathrm{old}}(x_i),\,\nabla v(x_i)\bigr\rangle
  \;=\;
  -\,h(x_i)\;-\; \normW{u^{\mathrm{old}}(x_i)}{R}^2,
  \qquad i=1,\dots,n.
\end{equation*}
Hence, by Theorem~\ref{lem:GaussNewton-vector}, a single Gauss-Newton step for \eqref{eq:minProblem} is obtained by solving the  problem
\begin{equation*}
  \min_{v \in \Hk} \; \|v\|_{\Hk}
  \quad \text{s.t.}\quad
  \begin{cases}
    \bigl\langle f(x_i)+g(x_i)\,u^{\mathrm{old}}(x_i),\,\nabla v(x_i)\bigr\rangle
      = - h(x_i) - \normW{u^{\mathrm{old}}(x_i)}{R}^2,
      & i=1,\dots,n,\\[2mm]
    v(x_i) \ge 0, & i=1,\dots,n,\\
    v(0) = 0, &
  \end{cases}
\end{equation*}
and, respectively, for \eqref{eq:minProblem2},
\begin{equation*}
  \min_{v \in \Hk} \; \|v\|_{\Hk}
  \quad \text{s.t.}\quad
  \begin{cases}
    \bigl\langle f(x_i)+g(x_i)\,u^{\mathrm{old}}(x_i),\,\nabla v(x_i)\bigr\rangle
      = - h(x_i) -  \normW{u^{\mathrm{old}}(x_i)}{R}^2 ,
      & i=1,\dots,n,\\[2mm]
    \alpha \,\|x_i\|^2 \;\le\; v(x_i) \;\le\; \beta\,\|x_i\|^2, & i=1,\dots,n,\\
    v(0) = 0. &
  \end{cases}
\end{equation*}
For the practical implementation, as discussed at the end of Subsection~\ref{sec:nonlinearOp}, it is admissible   to solve the problems above without the inequality constraints and to verify {a posteriori} that the new iterate satisfies them. This is particularly natural when approximating the OVF, which has no active constraints for appropriately chosen $\alpha$ and $\beta$. In our computations we therefore follow this approach: each Gauss-Newton step reduces to Algorithm~\ref{algo:PI}, with an inequality-feasibility check at every iteration. Furthermore, note that the linear independence of the functionals
 in $\lambda^{\mathrm{lin}}_{s^{\mathrm{old}}}$,
which follows from Assumption~\ref{ass:nonzero-centers} together with
Theorem~\ref{lem:GaussNewton-vector}, guarantees the well-posedness
 of the policy-iteration scheme.
However, Theorem~\ref{theo:WellPI} yields the same conclusion under
strictly weaker hypotheses. Hence, Algorithm~\ref{algo:PI} might be applicable
even in the absence of any connection to verification-based nonlinear
optimal recovery formulations for the OVF.

\section{Numerical experiments}\label{sec:numeric}

In this section we solve \eqref{eq:minProblem} and \eqref{eq:minProblem2} to approximate the OVF  using Algorithm~\ref{algo:PI}. An implementation of the proposed approach is available on GitHub.\footnote{\url{https://github.com/ehringts/RKHS-PI.git}} Each RKHS–PI update is accompanied by a {verification} step that checks whether all inequality constraints are satisfied at the centers. In practice, the treatment of the inequality constraints plays a comparatively minor role in our experiments: for the initial RKHS–PI policies specified below for each model problem, and for kernel expansions of sufficiently large size, all inequality constraints remain satisfied throughout the RKHS–PI sequence when choosing $\alpha=\tfrac{1}{2}\lambda_{\min}(P)$ and $\beta=2\lambda_{\max}(P)$, where $P$ is the positive definite solution of the ARE associated with the linearized model problem.\\
Before presenting the four model problems used to assess the schemes, we specify (i) the kernels defining the surrogate spaces, (ii) the strategy for selecting centers, and (iii) the error metrics employed in the evaluation.\\

\noindent \textbf{Kernel structure.}
To approximate  the OVF, we employ the radial kernels
\begin{align*}
    k_{LM}(x,y) &= \exp\!\bigl(-\gamma \,\|x-y\|\bigr)\,\bigl(1+\gamma \|x-y\|\bigr),\\
    k_{G}(x,y)  &= \exp\!\bigl(-(\gamma \,\|x-y\|)^{2}\bigr),
\end{align*}
with shape parameter $\gamma>0$ and $x,y\in\Omega\subset\mathbb{R}^{N}$. The kernel $k_{LM}$ is the linear Matérn kernel; its RKHS is norm–equivalent to a Sobolev space $W^{\frac{N+3}{2},2}(\Omega) \subset C^1(\Omega)$ for Lipschitz domain $\Omega \subset \mathbb{R}^N$ (see Corollary~10.48 in ~\cite{wendland2004}). The kernel $k_{G}$ is Gaussian, whose RKHS consists of real–analytic functions. Thus, the kernel choice determines whether the approximation space comprises real–analytic functions or functions that (nearly) meet the minimal requirement of being once differentiable.\\
We further consider product–form kernels
\begin{equation}\label{eq:product-kernel}
  k(x,y)\;=\;\bigl\langle x,y \bigr\rangle^{2}\,\tilde k(x,y),\qquad x,y\in\mathbb{R}^{N},
\end{equation}
where $\tilde k$ is positive definite. The quadratic factor implies
\[
  k(x,0)=0,
  \qquad
  \nabla_1 k(x,0)=0
  \quad\forall\,x\in\mathbb{R}^{N}.
\]
Consequently, any $\xi \in \Hk$ satisfies the origin constraints
\begin{equation}\label{eq:v-origin-constraints}
  \xi(0)=0,
  \qquad
  \nabla \xi(0)=0.
\end{equation}
This is advantageous in our setting because the OVF $v^*$ fulfills the same conditions at the origin,
\[
v^*(0)=0,
\qquad
\nabla v^*(0)=0,
\]
so the surrogate space automatically enforces the zero–at–the–origin verification constraints. Restricting the hypothesis space to functions obeying \eqref{eq:v-origin-constraints} can improve approximation quality when the target has the same structure. Moreover, under mild
conditions on \(\tilde k\), Assumption~\ref{ass:LinearInd}, needed for  Lemma~\ref{theo:KernelLinFunc}, holds for kernels of the form
\eqref{eq:product-kernel}. A precise statement will be given below.

\begin{theorem}\label{theo:quadKernel}
Let $N\in\mathbb{N}$ and let $\tilde k:\mathbb{R}^N\times\mathbb{R}^N\to\mathbb{R}$ be a translation–invariant strictly positive definite kernel of the form
\[
\tilde k(x,y)=\phi(x-y),
\]
where $\phi\in L^1(\mathbb{R}^{N},\mathbb{R})\cap C^{2}(\mathbb{R}^{N},\mathbb{R})$ with $(1+\Vert   \cdot   \Vert^2) \widehat{\phi}(\, \cdot \, ) \in L^1(\mathbb{R}^{N},\mathbb{R})$. Define
\[
k(x,y)\;=\;\langle x,y\rangle^{2}\,\tilde k(x,y),\qquad x,y\in\mathbb{R}^N,
\]
where $\langle \cdot,\cdot\rangle$ denotes the Euclidean inner product on $\mathbb{R}^N$.
Then $k$ is a positive definite kernel on $\mathbb{R}^N$.
Moreover, writing $\mathcal{H}_k(\mathbb{R}^N)$ for the RKHS of $k$, the family of evaluation and partial derivative evaluation functionals
\begin{equation}\label{eq:fam}
\bigl\{\delta_x\bigr\}_{x\in\mathbb{R}^N\setminus\{0\}}
\;\cup\;
\bigl\{\delta_x(\,\cdot\,)  \circ \partial^s\bigr\}_{x\in\mathbb{R}^N\setminus\{0\},\; s=1,\dots,N}
\end{equation}
is linearly independent in the dual space $\mathcal{H}_k(\mathbb{R}^N)'$.
\end{theorem}
\begin{proof}
    See Appendix \ref{sec:quadKernelProof}.
\end{proof}

\noindent
Regarding the structure of the corresponding RKHS, Aronszajn's seminal work~\cite{aronszajn1950theory} provides the characterization. In particular, by Theorem~II therein, the associated RKHS can be written as
\begin{align}   \label{eq:produktkernelRKHS}
\Hk = \left\{ \left\langle x ,  F(x)\,  x \right\rangle \; : \; F(x) \in \left(\mathcal{H}_{\tilde{k}}(\Omega)\right)^{N \times N} \right\}.
\end{align}
The two concrete instances used in our experiments for approximating the OVF are
\begin{align*}
    k_{LM,Q}(x,y) &= \exp\!\bigl(-\gamma \,\|x-y\|\bigr)\,\bigl(1+\gamma \|x-y\|\bigr)\,\langle x,y\rangle^{2},\\
    k_{G,Q}(x,y)  &= \exp\!\bigl(-(\gamma \,\|x-y\|)^{2}\bigr)\,\langle x,y\rangle^{2},
\end{align*}
where $\gamma>0$ and $x,y\in\Omega\subset\mathbb{R}^{N}$. Their RKHSs are characterized by \eqref{eq:produktkernelRKHS}; in particular, the RKHS of $k_{G,Q}$ still consists of real–analytic functions.\\
For the PI, we present numerical experiments with two enforcement strategies for the zero-at-the-origin verification condition: (i) an {explicit} treatment, implemented as a functional constraint, and (ii) an {implicit} treatment induced by the kernel construction, i.e. the product kernels \(k_{LM,Q}\) or \(k_{G,Q}\).\\

\noindent\textbf{Center selection.}
Next, we  describe the selection of centers $X_n$ for the optimal recovery problems \eqref{eq:minProblem} and \eqref{eq:minProblem2} used to approximate the OVF via the RKHS–PI. 
The convergence results in Theorems~\ref{thm:convergence-quadratic} and \ref{tho:OptRecovAna} require collocation sets that are (approximately) dense in \(\Omega\). 
While a sufficiently fine candidate grid \(\Omega_G\subset\Omega\) approximates this density condition, it typically leads to severe ill-conditioning of the associated kernel matrices and unnecessary computational cost. 
To balance accuracy and numerical stability, we adopt a greedy enrichment strategy in the first RKHS–PI iteration:  Starting from $X_0=\emptyset$, the set of centers is enlarged iteratively by
\begin{equation*}
    x_{n+1}\in\arg\max_{x\in\Omega_G\setminus X_n}\,\nu_{s_v^n}(x),
    \qquad
    X_{n+1}:=X_n\cup\{x_{n+1}\},
\end{equation*}
where $\nu_{s_v^n}:\Omega\to\mathbb{R}$ is a target-dependent selection functional that depends on the current surrogate $s_v^n$ built from the centers $X_n$.
We adopt a residual-based selection criterion: 
With the current surrogate $s_v^n$ of $v$, we select points by the pointwise magnitude of the relative GHJB residual
\begin{equation*}
  \nu_{\mathrm{PI},\,s_v^n}(x)
  \;:=\;
  \left|\,
  \frac{\text{GHJB}(s_v^n,u_0,x)}{h(x)+\normW{u_0(x)}{R}^2}
  \,\right|,
\end{equation*}
where $u_0$   denotes the initial policy. Note that $\nu_{\mathrm{PI},\,s_v^n}(x)$ will be only evaluated for $x \in \Omega \setminus \{0\}$. \\
\noindent The use of target-dependent, residual-driven greedy sampling for linear PDE approximation is discussed in \cite{wenzel2025adaptive}, where dimension-independent convergence rates are established. For this reason we prefer adaptive center
selection over grid–based designs, which are impractical for two of our model
problems due to the curse of dimensionality.

\noindent \textbf{Error measures.} All quantitative assessments are performed on a training set
 and a test set, i.e.,
\[
  \Omega_{\text{\tiny train}}
  := \bigl\{x_{\text{\tiny train}}^{(1)},\ldots,x_{\text{\tiny train}}^{(n_{\text{\tiny train}})}\bigr\}\subset\Omega\setminus \{0\}
\quad\text{and}\quad
  \Omega_{\text{\tiny test}}
  := \bigl\{x_{\text{\tiny test}}^{(1)},\ldots,x_{\text{\tiny test}}^{(n_{\text{\tiny test}})}\bigr\}\subset\Omega\setminus \{0\},
\]
whose specifications are given with each model problem.

\smallskip
\noindent
(i) \emph{Residual error on the training set.}
We monitor the maximal GHJB residual for RKHS–PI (see Algorithm \ref{algo:PI}) over the training set during the greedy center–selection procedure in the initial PI  iteration:
\begin{align*}
  \mathrm{Res\text{-}GHJB}
  := \max_{x\in\Omega_{\text{\tiny train}}}
      \bigl|\,\nu_{\mathrm{PI},\,s_v^{\,n}}(x)\,\bigr|. 
\end{align*}
For a fixed expansion size (number of kernel centers), we select both the kernel  and the shape parameter $\gamma>0$ by minimizing the final training residual defined above.

\smallskip
\noindent
(ii) \emph{True error on the test set.}
For a reference solution, we report the relative $\ell^{2}$ error on the test set,
\[
  \mathrm{Error\text{-}PI}
  :=
  \sqrt{
    \frac{\displaystyle\sum_{i=1}^{n_{\text{\tiny test}}}
          \bigl|\,v^*\!\bigl(x_{\text{\tiny test}}^{(i)}\bigr)-s_v^{\,n}\!\bigl(x_{\text{\tiny test}}^{(i)}\bigr)\bigr|^{2}}
         {\displaystyle\sum_{i=1}^{n_{\text{\tiny test}}}
          \bigl|\,v^*\!\bigl(x_{\text{\tiny test}}^{(i)}\bigr)\bigr|^{2}} 
  }.
\]
For only two of the model problems the exact OVF is known analytically across the entire domain. For the remaining problems, we approximate $v^*$ on $\Omega_{\text{\tiny test}}$ via Pontryagin’s maximum principle (using the implementation of~\cite{azmi2021optimal}) on a finite-time horizon with terminal cost and a sufficiently large horizon. This procedure approximates the infinite-horizon OVF and introduces a small error in the reported true error.

\subsection{Academic toy example}
We consider a two-dimensional controlled system with a single (scalar) control input. The infinite-horizon problem admits the true OVF
\begin{align*}
    v^*(x_1,x_2) = \frac{1}{2}\,x_1^{2} + x_2^{2}.
\end{align*}
This example, adapted from \cite{Vrabie2009}, is a special case of \eqref{eq:MPInfinit} with
\begin{gather*}
    h(x_1,x_2) = x_1^{2} + x_2^{2}, 
    \qquad R = 1,\\[2mm]
    f(x_1,x_2) = 
    \begin{bmatrix}
        -x_1 + x_2 \\[1mm]
        -\tfrac{1}{2}\,(x_1 + x_2) + \tfrac{1}{2}\,x_2 \sin^{2}(x_1)
    \end{bmatrix},
    \qquad
    g(x_1,x_2) = 
    \begin{bmatrix}
        0 \\[1mm]
        \sin(x_1)
    \end{bmatrix}.
\end{gather*}
Following \cite{Vrabie2009}, we initialize the RKHS–PI with the feedback law 
$
    u_0(x_1,x_2) = -\tfrac{3}{2}\,\sin(x_1 + x_2).
$ 
The computational domain is $\Omega = [-1,1]^2$. For the numerical experiments, the training set is
$ 
    \Omega_{\text{\tiny train}} =  G \times G \subset \Omega,
$
where $G$ is a one-dimensional uniform grid with $|G|=100$ nodes on $[-1,1]$. The test set
 $\Omega_{\text{\tiny test}}$ consists of $100$ points drawn independently from the uniform distribution on $\Omega$.\\
Figure~\ref{fig:toyExample} summarizes the results. Using either the kernel \(k_{{G}}\) or \(k_{{G,Q}}\), we select the shape parameter \(\gamma\) by minimizing the $\mathrm{Res\text{-}GHJB}$ loss in the first PI  step, which yields  \(\gamma  = \sqrt{1.7}\) for both kernels. With this choice, the RKHS--PI algorithm exhibits rapid convergence:
 the relative GHJB residual on the training set (cf.\ the error metrics defined
above) drops below \(10^{-5}\) within \(200\)  selected centers, and the relative
test error falls below \(10^{-6}\) within \(4\) RKHS–PI iterations. On this problem,
the structure–aware product kernels do not yield a noticeable advantage over
their standard counterparts.

\begin{figure}[htbp]  
    \centering         
    \includegraphics[width=0.9\textwidth]{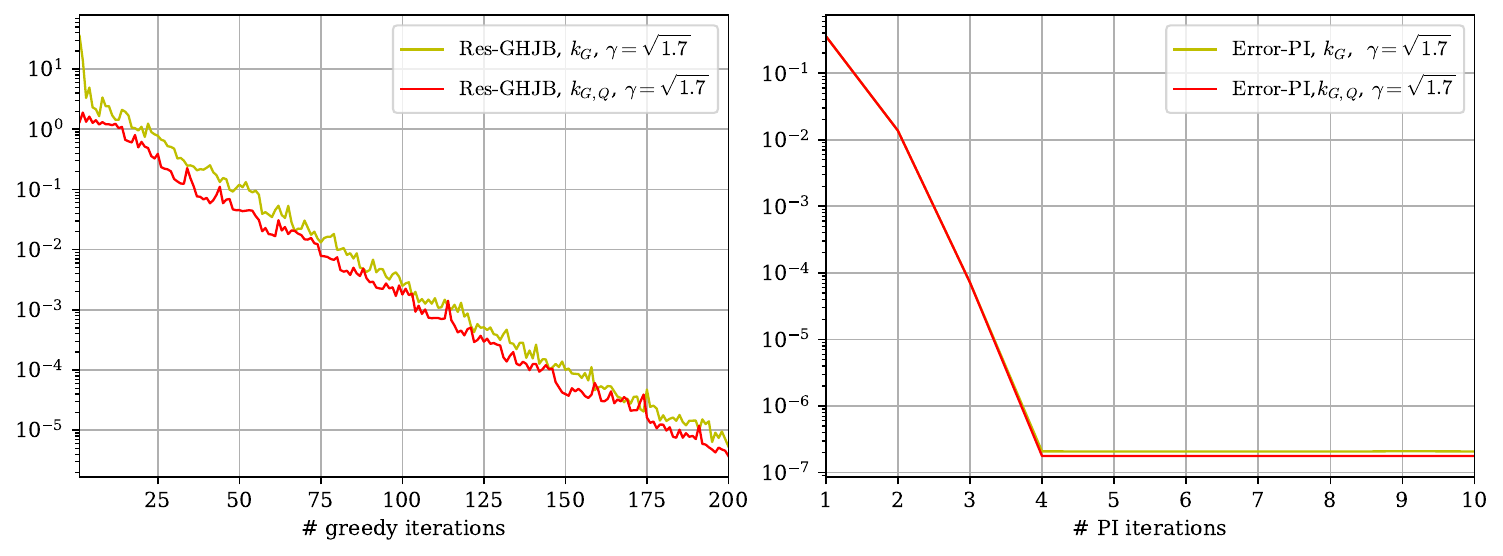}  
    \caption{Academic toy example: The training error of the initial RKHS–PI iteration  plotted against the number of selected centers (left). The True-Error for the RKHS–PI   over the number of  iterations (right).}
    \label{fig:toyExample}  
\end{figure}

\subsection{Van der Pol oscillator}
The second model problem analyzed is the Van der Pol oscillator, a two-dimensional system with a single control input. This problem can be represented as a specific instance of  \eqref{eq:MPInfinit}  using the following definitions: 
\begin{gather*}
    h(x_1,x_2) = x_1^2 + x_2^2, \;\;
    R = \frac{1}{10}, \\
    f(x_1,x_2) = \begin{bmatrix}
         x_2 \\
        -x_1 +  x_2(1-x_1^2) 
    \end{bmatrix}, \text{ and }\,
    g(x_1,x_2)= \begin{bmatrix}
        0 \\
        1
    \end{bmatrix}
\end{gather*}
The initial controller for the RKHS–PI is defined as $$u_0(x_1,x_2) = -\frac{1}{R}g(x_1,x_2)^{\top} P_{\text{\tiny{VP}}} \begin{bmatrix}
        x_1 \\
        x_2
    \end{bmatrix},$$ where $P_{\text{\tiny{VP}}}$ is the solution to the ARE for the linearized version of the problem. This linear quadratic regulator (LQR) approach involves replacing $f$ with its linearized form around the origin:
    \begin{align*}
        \tilde{f}(x_1,x_2)= \begin{bmatrix}
        0 & 1 \\
        -1 & 1
    \end{bmatrix} \begin{bmatrix}
        x_1 \\
        x_2
    \end{bmatrix}
    \end{align*}
The computational domain for this problem is set as $\Omega = [-1,1]^2$.  Consistent with the earlier example, the training domain is defined as $\Omega_{\text{\tiny{train}}} = G \times G$ and the test set
 $\Omega_{\text{\tiny test}}$ consists of $100$ points drawn independently from the uniform distribution on $\Omega$. Note that for this model problem, the OVF is evaluated using the PMP. The shape parameters are set to $\gamma = \sqrt{1.7}$ for $k_{{G}}$ and  $\gamma = \sqrt{1.1}$ for $k_{{G,Q}}$, respectively. Figure~\ref{fig:VanDerPol} summarizes the results. The results are favorable:  For this problem, it is beneficial to use the structure-aware kernel in the standard PI; it improves the true error by more than one order of magnitude compared with the version without the structure-aware kernel.

\begin{figure}[htbp]  
    \centering         
    \includegraphics[width=0.9\textwidth]{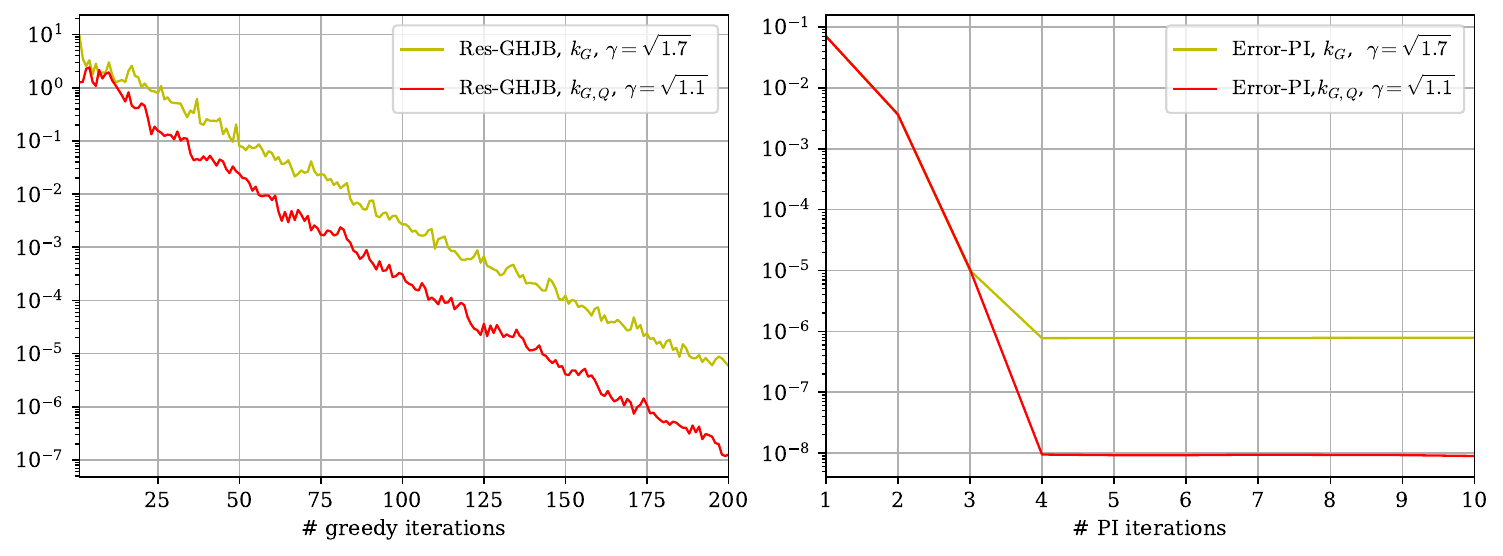}  
    \caption{Van der Pol oscillator:  The training error of the initial RKHS–PI iteration  plotted against the number of selected centers (left). The True-Error for the RKHS–PI over the number of  iterations (right).}
    \label{fig:VanDerPol}  
\end{figure}

\subsection{Linear heat equation with  Dirichlet boundary conditions}
The next model problem under consideration arises from a semi-discretized linear heat equation. The governing equation is given by
\begin{align*}
\dot{\vartheta}(\xi,t) =   \Delta \vartheta(\xi,t)+\sum_{i=1}^4\Xi_i(\xi) u_i(t)
\end{align*}
for $(\xi,t) \in I \times [0,\infty)$ with the boundary conditions
\begin{align}\label{eq:boundCon}
\vartheta(\xi,t) = 0  \text{ for }(\xi,t) \in \partial I \times  [0,\infty), \quad
\vartheta(\xi,0) =  \vartheta_0(\xi)  \text{ for }\xi \in  I
\end{align}
and domain $I := (0,1)$. Here, $\Delta$ is the Laplace operator. The  functions $\Xi_i(\xi)$ are indicator functions: $\Xi_1$ is   equal to 
$1$ on the interval  $[0.1, 0.2]$, $\Xi_2$ on $[0.3, 0.4]$, $\Xi_3$ on $[0.6, 0.7]$, and $\Xi_4$ on $[0.8, 0.9] $. The controller dimension of the system is four ($M = 4$), reflecting that the control inputs are independent of the spatial variable. To formulate an ODE-constrained optimization problem, as pursued in the present paper, an appropriate semi-discretization of the system must first be performed. We employ a semi-discretization approach using the Kansa-type collocation method (see \cite{kansa1990multiquadrics}), where the surrogate solution for the PDE is of the form
\begin{align*}
    s_{\vartheta}(\xi) = \sum_{j=1}^N (\mathbf{x}(t))_j k(\xi_j,\xi)
\end{align*}
with equidistant centers $\xi_j = \tfrac{j}{N+1}$ for $j=1,...,N$ and  a kernel function $k: I \times I \longrightarrow \mathbb{R}$ satisfying the Dirichlet boundary conditions:
\begin{align*}
   k(\xi_j,0) = k(\xi_j,1) = 0 \text{ for all } j=1,...,N
\end{align*}
The time-dependent coefficients $(\mathbf{x}(t))_j$ for $j=1,...,N$ are determined such that the PDE is satisfied point-wise at the collocation points $\{\xi_l\}_{l=1}^N$: 
\begin{align*}
\sum_{i=1}^N (\dot{\mathbf{x}}(t))_i k(\xi_i,\xi_l) =   \sum_{i=1}^N (\mathbf{x}(t))_i  \Delta_2 k(\xi_i,\xi_l)+\sum_{i=1}^4\Xi_i(\xi_l) u_i(t) \text{ for all } l=1,...,N
\end{align*}
This results in a semi-discretized version of the PDE, which is an ODE of the coefficients  $\mathbf{x}(t)$ of the form
\begin{gather}\label{eq:heatODE}
 \dot{\mathbf{x}}(t) = \underbrace{K^{-1} K_{\Delta} }_{=:A}\mathbf{x}(t)+ \underbrace{ K^{-1}\begin{bmatrix}
    b_1 &
   b_2  &
    b_3 &
    b_4
    \end{bmatrix}}_{=:B} \underbrace{ \begin{bmatrix}
    \mathbf{u}_1(t) \\
    \mathbf{u}_2(t) \\
    \mathbf{u}_3(t) \\
    \mathbf{u}_4(t) 
    \end{bmatrix} }_{=: \mathbf{u}(t)} \text{ and }\mathbf{x}(0) = x_0  .
\end{gather}
with Gram matrices $(K)_{i,j} = k(\xi_i,\xi_j)$, $(K_\Delta)_{i,j} = \Delta_2 k(\xi_i,\xi_j)$  for $i,j=1,...,N$
and  the vectors $b_i \in \mathbb{R}^{N}$  representing the discretized forms of the indicator functions  $\Xi_i$. As the kernel function, we utilized 
\begin{align*}
    k(\xi,\xi') = e^{-3000 \Vert \xi - \xi' \Vert^2}\; \xi (1-\xi) \xi' (1-\xi')
\end{align*}
and the number of collocation points is set to $N=50$ for subsequent numerical experiments.
The ODE system \eqref{eq:heatODE} possesses a stable equilibrium state corresponding to the constant zero function. To optimally steer the system toward this equilibrium, we formulate the following OCP:
\begin{gather}
\min_{\mathbf{u} \in \mathcal{U}_{\infty}} \inte{0}{\infty}{t}{ \norm{\mathbf{x}(t)}_{2} ^2 + \frac{1}{100} \cdot \norm{\mathbf{u}(t)}_{2}^2} \label{eq:LHE1}   \\ 
\text{s.t. } \dot{\mathbf{x}}(t) = A\mathbf{x}(t)+ B \mathbf{u}(t) \text{ and }\mathbf{x}(0) = x_0   \label{eq:LHE2} 
\end{gather}
Notably, the trivial control input $\mathbf{u}(t) \equiv 0$ is admissible. Thus, we set  $u_0(x) \equiv 0$.  The computational domain is specified as $\Omega = [0,10]^{50}$. Due to the curse of dimensionality, a regular grid as used in previous model problems is infeasible. Instead, we define a training set $\Omega_{\text{\tiny{train}}}$ consisting of $10^5$  uniformly distributed samples  in $\Omega$. For the LQR problem \eqref{eq:LHE1}--\eqref{eq:LHE2}, the true OVF is  given by $$v(x) = \langle x, P_{\text{\tiny{LHE}}} x\rangle$$
where $P_{\text{\tiny{LHE}}}$ is the positive definite solution to the corresponding ARE. Here too, the test set $\Omega_{\text{\tiny{test}}}$ comprises $100$
uniformly distributed samples  in  $\Omega$. The shape parameters are set to $\gamma = \sqrt{6} \cdot 10^{-5}$ for $k_{{G}}$ and $\gamma = 5 \cdot 10^{-8}$ for $k_{{LM,Q}}$. Due to its superior performance in the structure-aware setting, we used $k_{{LM,Q}}$ for that experiment. \\
Figure~\ref{fig:Lin} summarizes the results. Owing to the high state dimension of $50$, the training errors decrease initially much more slowly, which in turn leads to a worse approximation of the OVF. For the RKHS–PI it is around $2\%$ after two iterations for the Gaussian kernel.\\
However, this does not apply to the  RKHS–PI when we use the structure-aware kernel. After approximately $1275$ selected centers, one observes that the training error drops rapidly. This number aligns with the degrees of freedom of the positive definite matrix $P_{\text{\tiny{LHE}}} \in \mathbb{R}^{50 \times 50}$. Consistent with the RKHS structure induced by the structure-aware kernel $k_{LM,Q}$ (see \eqref{eq:produktkernelRKHS}), we expect that the RKHS–PI scheme has effectively approximated the positive definite  solution of the ARE.

\begin{figure}[htbp]  
    \centering         
 \includegraphics[width=0.9\textwidth]{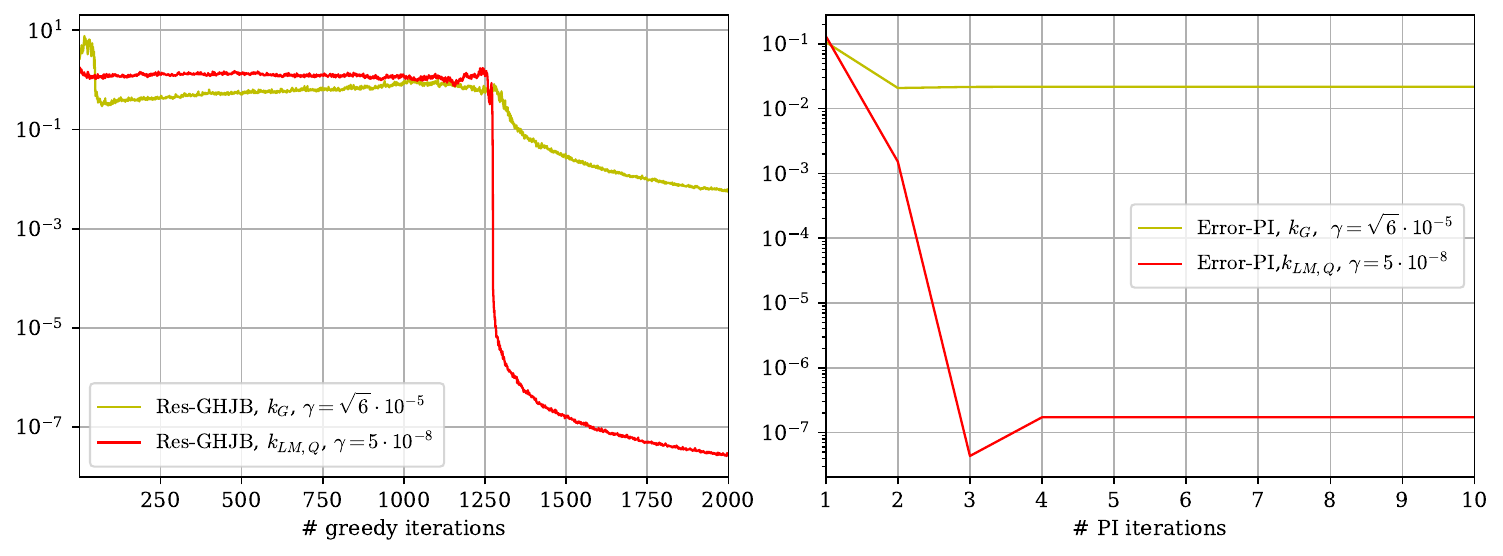} 
    \caption{Linear heat equation: The training error of the initial RKHS–PI iteration  plotted against the number of selected centers (left). The True-Error for the RKHS–PI   over the number of  iterations (right).}
    \label{fig:Lin}  
\end{figure}

\subsection{Nonlinear heat equation with  Dirichlet boundary conditions}
The final model problem extends the previously considered linear heat equation to a nonlinear version of the Zeldovich type. The governing equation is given by
\begin{align*}
\dot{\vartheta}(\xi,t) =   \Delta \vartheta(\xi,t)+\vartheta(\xi,t)^2-\vartheta(\xi,t)^3+\sum_{i=1}^4\Xi_i(\xi) u_i(t)
\end{align*}
for $(\xi,t) \in I \times [0,\infty)$ with the boundary conditions
 as in \eqref{eq:boundCon}. To derive a semi-discretized version of the problem, the same collocation methods as in the previous section are employed. The corresponding control problem is formulated using the same cost functional, leading to the OCP
\begin{gather*}
\min_{\mathbf{u} \in \mathcal{U}_{\infty}}\inte{0}{\infty}{t}{ \norm{\mathbf{x}(t)}_{2} ^2 + \frac{1}{100} \cdot \norm{\mathbf{u}(t)}_{2}^2 }     \\ 
\text{s.t. } \dot{\mathbf{x}}(t) = A\mathbf{x}(t)+  K^{-1}\left( (K\mathbf{x}(t))^{\odot 2}-(K\mathbf{x}(t))^{\odot 3} \right) +  B \mathbf{u}(t) \text{ and }\mathbf{x}(0) = x_0,    
\end{gather*}
denoting,  for a vector $v\in\mathbb{R}^n$ and $p\in\mathbb{N}$, by $v^{\odot p}$ the Hadamard (componentwise) $p$-th power.
 The initialization for RKHS–PI  follows the same choices as in the previous model problem. The computational domain is set as $\Omega = [0,10]^{50}$, with a training set $\Omega_{\text{\tiny{train}}}$ comprising $10^5$ uniformly distributed samples within $\Omega$. Unlike the previous model problem, where the true OVF was known, OVF samples must be approximated using open-loop control for this nonlinear case. The test set $\Omega_{\text{\tiny{test}}}$ is constructed, consisting of $100$ uniformly distributed samples in $\Omega$. The shape parameters are set to $\gamma = \sqrt{6} \cdot 10^{-5}$ for $k_{{G}}$ and $\gamma = 4 \cdot 10^{-8}$ for $k_{{LM,Q}}$.  \\
The results in Figure~\ref{fig:nonlin} are qualitatively  very similar to those of the previous model problem. Here, too, the RKHS–PI with the structure-aware kernel outperforms the  other variant. However, note that for this nonlinear problem the OVF is not known to be quadratic, even though the RKHS induced by \eqref{eq:produktkernelRKHS} may still be a beneficial ansatz space.

\begin{figure}[htbp]  
    \centering         
    \includegraphics[width=0.9\textwidth]{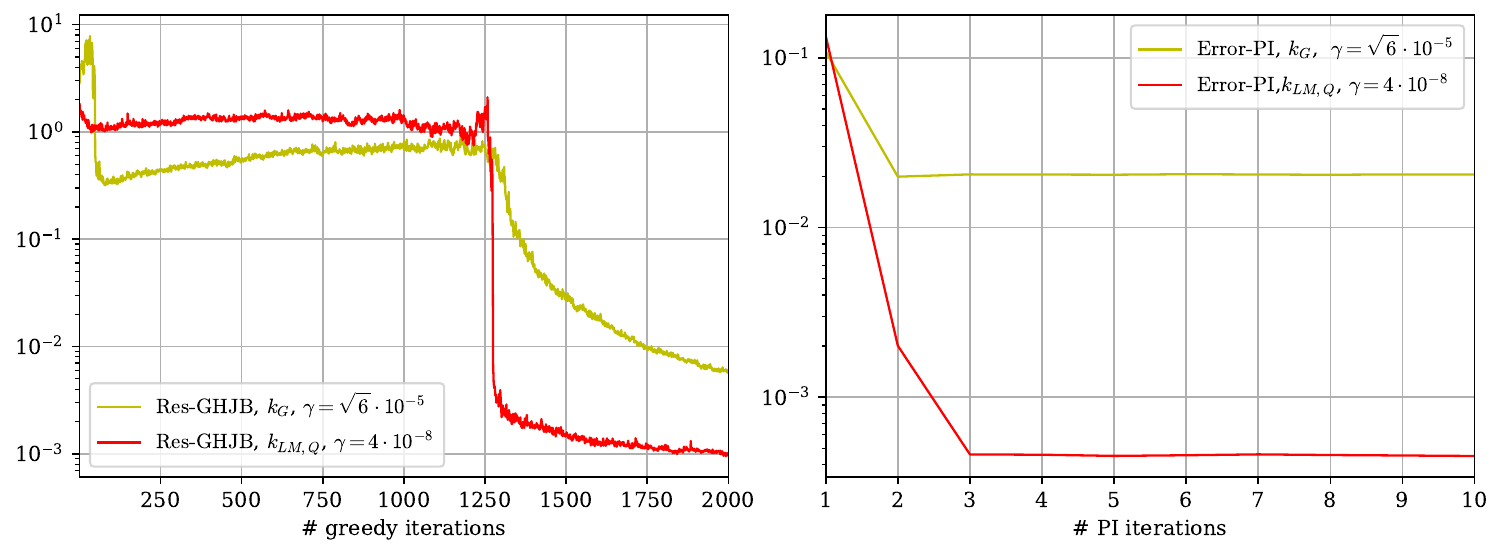}  
    \caption{Nonlinear heat equation:  The training error of the initial RKHS–PI iteration   plotted against the number of selected centers (left). The True-Error for the RKHS–PI   over the number of  iterations (right).}
    \label{fig:nonlin}  
\end{figure}

\section{Conclusion and Outlook}\label{sec:conclusion}
\noindent 
We presented a verification–based methodology for approximating OVFs on bounded domains within an RKHS framework. Under a real-analyticity assumption, Corollary~\ref{coro:VerificAnalytic} provides {global} identification of the OVF on $\Omega$. Under milder assumptions -- namely, uniform quadratic lower and upper bounds -- Corollary~\ref{coro:VerificConstant} guarantees {local} identification. 
These verification conditions are embedded in an abstract nonlinear optimal recovery formulation that simultaneously handles both equality and inequality constraints. A finite-dimensional reduction (Theorem~\ref{theo:finiteDimNonlinear}) yields an exact finite-dimensional program. For both the analytic and the quadratic regimes, we prove convergence of the resulting schemes provided the collocation points become dense in $\Omega$ (Theorems~\ref{tho:OptRecovAna} and~\ref{thm:convergence-quadratic}). \\
A notable limitation of this verification–based approximation framework is the standing assumption that the OVF is continuously differentiable, so that it satisfies the HJB equation in the classical (strong) sense. A natural direction for future work is to relax this requirement and consider viscosity solutions of the HJB together with their associated verification conditions; see, e.g., \cite{bardi1997optimal}. It would then be interesting to investigate whether such viscosity–solution–based conditions can also be cast within the same abstract nonlinear optimal recovery framework. \\
\noindent For the numerical solution of the nonlinear optimal–recovery problems developed from verification conditions, we propose a Gauss–Newton linearization that preserves the inequality constraints and, at each iteration, yields a linear optimal–recovery subproblem. This construction is algorithmically equivalent to RKHS–PI with a verification check for OVF approximation. Even without this verification check, the well–posedness of the policy-evaluation step in RKHS–PI holds under mild assumptions (Theorem~\ref{theo:WellPI}). By contrast, global convergence of the overall PI procedure is not guaranteed a priori; in our analysis it will hinge on the convergence properties of the Gauss–Newton method.  A complementary research direction is to establish convergence of RKHS–PI from the standpoint of approximate PI, explicitly accounting for function-approximation and evaluation errors. \\
Concerning the Gauss–Newton algorithm, incorporating line search or damping strategies~\cite{nocedal2006numerical,deuflhard1979affine,chen2005convergence,ferreira2011local} is a natural next step. For this, it would be particularly relevant to analyze how an adaptive step size affects the associated finite–dimensional (linear) optimal–recovery subproblems and their conditioning.\\
\noindent Regarding the numerical observations, across all the model problems considered, the RKHS–PI scheme converged rapidly:
Training residuals decayed quickly and test errors were small after a few
iterations. In our experiments, inequality constraints were typically satisfied
throughout the RKHS–PI sequence when initialized with a suitably chosen stabilizing
policy; thus, the verification check played a minor practical role.\\
\noindent In summary, RKHS–based policy iteration offers a practically viable route to verification–driven OVF approximation. The results herein establish well–posedness and convergence in idealized regimes and open several avenues for sharpening both theory and algorithms.\\

\section*{Acknowledgements}
 The authors appreciate and acknowledge Matthias Baur for his helpful comments and insights on this manuscript.  Funded by Deutsche Forschungsgemeinschaft (DFG, German Research Foundation) under Project No. 540080351 and Germany’s Excellence Strategy -- EXC 2075 -- 390740016. We acknowledge support from the Stuttgart Center for Simulation Science (SimTech).

\appendix
\section{Proof of the statements of Section \ref{sec:VerificationCon}}
\subsection{Proof of the verification of optimality Theorem \ref{thm:VerOpt}}\label{sec:appoptVer}
The proof begins by establishing three properties -- proved in the subsequent three steps -- that the OVF \(v^*\) and the corresponding closed-loop dynamics
\begin{align}\label{eq:closedLoopProof}
\dot{\mathbf{x}}^*(t;x)
  = f\!\bigl(\mathbf{x}^*(t;x)\bigr)
    + g\!\bigl(\mathbf{x}^*(t;x)\bigr)\,
      u^{*}\!\bigl(\mathbf{x}^*(t;x)\bigr),
\end{align}
with the feedback law
\[
  u^{*}(x) \;=\; -\tfrac{1}{2}\,R^{-1}\,g(x)^{\!\top}\,\nabla v^{*}(x),
\]
satisfy. 
Since the argument for these properties relies only on (i) \(v^*\) being positive semidefinite -- i.e., \(v^*(x)\ge 0\) for all \(x\in\Omega\) -- which follows from the definition of \(v^*\) and the positive definiteness of \(h\) and \(R\), and on (ii) \(v^*\) satisfying the HJB equations \eqref{eq:HBJ1}–\eqref{eq:HBJ2} under the assumption \(v^*\in C^{1}(\Omega,\mathbb{R})\), the same three properties also hold for the candidate OVF \(v\) and its closed-loop dynamics
\begin{align}\label{eq:clCOVF}
  \dot{\mathbf{x}}_{u_v}(t;x)
     = f\bigl(\mathbf{x}_{u_v}(t;x)\bigr)
       + g\bigl(\mathbf{x}_{u_v}(t;x)\bigr)\,
         u_v\bigl(\mathbf{x}_{u_v}(t;x)\bigr),
  \qquad
  \mathbf{x}_{u_v}(0;x)=x.
\end{align}
For brevity and clarity, we present the argument only for \(v^*\).

 \medskip
 \noindent  
\textbf{Step 1 (Positive definiteness of $v^*$).}
Fix any $x\in\Omega\setminus\{0\}$. By continuity of the closed-loop trajectory
$t\mapsto \mathbf{x}^*(t;x)$ and openness of $\Omega$, there exists
$t_{x}^*>0$ such that $\mathbf{x}^*(t;x)\in\Omega$ and
$\mathbf{x}^*(t;x)\neq 0$ for all $t\in[0,t_{x}^*]$.
Since $v^*\in C^1(\Omega,\mathbb{R})$ satisfies the HJB equation, integrating from $0$ to $t_{x}^*$ yields
\begin{align*}
    v^*(x)
& =v^*\bigl(\mathbf{x}^*(t^*_{x};x)\bigr) -\inte{0}{t_{x}^*}{t}{\left\langle \nabla v^*(\mathbf{x}^*(t;x)),\, f(\mathbf{x}^*(t;x))+g(\mathbf{x}^*(t;x))\,u^{*}(\mathbf{x}^*(t;x)) \right\rangle}\\
&= v^*\bigl(\mathbf{x}^*(t^*_{x};x)\bigr) + \inte{0}{t_{x}^*}{t}{h\bigl(\mathbf{x}^*(t;x)\bigr)
+ \|u^*\bigl(\mathbf{x}^*(t;x)\bigr)\|_{R}^{2}} \geq \inteNo{0}{t_{x}^*}{t}{h\bigl(\mathbf{x}^*(t;x)\bigr)} 
\end{align*}
using $v^*\ge 0$.
Because $h$ is positive definite and $\mathbf{x}^*(t;x)\neq 0$ for all
$t\in[0,t_{x}^*]$, the integrand is strictly positive on a set of positive
measure, hence the integral (and therefore $v^*(x)$) is strictly positive.
Together with $v^*(0)=0$, this shows that $v^*$ is positive definite on $\Omega$.
 
 \medskip
 \noindent 
\textbf{Step 2 (Construction of bounded forward-invariant sets).}
 A set \(S\subset\Omega\) is {forward-invariant} for the ODE
\(\dot{\mathbf{x}}(t)=\ell(\mathbf{x}(t))\), with \(\ell\in C(\Omega,\mathbb{R}^N)\),
if for every \(x\in S\) the corresponding solution \(\mathbf{x}(\cdot;x)\) exists for all
\(t\ge 0\) and satisfies \(\mathbf{x}(t;x)\in S\) for all \(t\ge 0\).

\medskip
\noindent We construct a bounded forward-invariant neighborhood of the origin for the closed-loop
optimal dynamics $\mathbf{x}^{*}(t;x)$. Since \(\Omega\) is open and
\(0\in\Omega\), there exists \(r>0\) such that \(\cl{B_r(0)}\subset\Omega\) (denoting with $\operatorname{cl}$  the closure of a set).
By continuity and positive definiteness of \(v^{*}\), the minimum of \(v^{*}\) on the
compact set \(\{y\in\Omega:\|y\|=r\}\) is attained and strictly positive; define
\[
c\;:=\;\min_{\{y\in\Omega:\,\|y\|=r\}} v^{*}(y)\;>\;0,
\qquad
\Omega_{v^{*}}\;:=\;\bigl\{\,y\in\cl{B_r(0)}:\; v^{*}(y)<c\,\bigr\}.
\]
Clearly, \(\Omega_{v^{*}}\) is bounded and \(\cl{\Omega_{v^{*}}}\subset\cl{B_r(0)}\subset\Omega\).
Next, suppose, towards a contradiction, that \(\Omega_{v^{*}}\) is not forward
invariant. Then there exist \(x\in\Omega_{v^{*}}\) and a first exit time
\[
t'\;:=\;\inf\{\,t>0:\;\mathbf{x}^{*}(t;x)\notin \Omega_{v^{*}}\,\}\;>\;0,
\]
such that by continuity \(\mathbf{x}^{*}(t';x)\in \partial\Omega_{v^{*}}\subset\cl{B_r(0)}\).
Integrating the HJB identity
along the trajectory \(\mathbf{x}^{*}(\cdot;x)\)  yields
\[
  c \;>\; v^*(x)
  \;=\;
  v^*\bigl(\mathbf{x}^*(t';x)\bigr)
  + \inte{0}{t'}{t}{
       h\!\bigl(\mathbf{x}^*(t;x)\bigr)
       + \normW{u^*\!\bigl(\mathbf{x}^*(t;x)\bigr)}{R}^2
     }
  \;\ge\;
  v^* \bigl(\mathbf{x}^*(t';x)\bigr).
\]
Now, two cases are possible:
 If $\mathbf{x}^*(t';x)\in B_r(0)$ and since 
        $v^* \bigl(\mathbf{x}^*(t';x)\bigr)<c$, we have
        $\mathbf{x}^*(t';x)\in\Omega_{v^*}$, contradicting the choice of~$t'$.
   If
        $\mathbf{x}^*(t';x)\in\partial B_r(0)$, then
        \[
          v^* \bigl(\mathbf{x}^*(t';x)\bigr)\;\ge\;
          \min_{\{y\in\Omega\mid\|y\|=r\}} v^*(y) \;=\; c,
        \]
        which contradicts $v^*\!\bigl(\mathbf{x}^*(t';x)\bigr)<c$.
Therefore \(\Omega_{v^{*}}\) is forward invariant. Since \(0\in\Omega_{v^{*}}\),
\(c>0\), and \(v^{*}\) is positive definite and continuous, \(\Omega_{v^{*}}\) contains
a  neighborhood of the origin.

 \medskip
 \noindent 
\textbf{Step 3 (Asymptotic stability of the closed-loop dynamics).}
Let \(x\in\Omega_{v^*}\). By Step~2, the trajectory \(t\mapsto \mathbf{x}^*(t;x)\) of
\eqref{eq:closedLoopProof} with initial condition \(x\) remains in the forward-invariant,
bounded set \(\Omega_{v^*}\) for all \(t\ge 0\). We now prove that
\begin{align}\label{eq:asymEq}
\lim_{t\to\infty}\|\mathbf{x}^*(t;x)\|=0 .
\end{align}
We use Barbalat’s lemma (see, e.g., \cite{farkas2016variations}): if
\(p\colon[0,\infty)\to\mathbb{R}\) is uniformly continuous and
\(\inteNo{0}{\infty}{t}{p(t)}<\infty\), then \(\lim_{t\to\infty}p(t)=0\).\\
Because \(v^{*}\) satisfies the HJB equation on \(\Omega_{v^*}\subset\Omega\), integrating
 the HJB identity along the closed-loop trajectory \eqref{eq:closedLoopProof} from \(0\) to an arbitrary horizon \(T>0\) yields, for every
\(x\in\Omega_{v^*}\),
\begin{align*}
v^{*}(x)
= v^{*}\!\bigl(\mathbf{x}^*(T;x)\bigr)
+ \inte{0}{T}{t}{
    h\!\bigl(\mathbf{x}^*(t;x)\bigr)
    + \|u^{*}\!\bigl(\mathbf{x}^*(t;x)\bigr)\|_{R}^{2}
  }.
\end{align*}
Letting \(T\to\infty\) and using \(v^{*}\!\bigl(\mathbf{x}^*(T;x)\bigr)\ge 0\) gives
\begin{equation}\label{eq:finite-cost}
\inteNo{0}{\infty}{t}{ h\left(\mathbf{x}^*(t;x)\right)}
\;\le\; v^{*}(x) \;<\; \infty .
\end{equation}
To apply Barbalat’s lemma to \(p(t):=h \bigl(\mathbf{x}^*(t;x)\bigr)\), we must show that
\(p\) is uniformly continuous. For all \(t\ge 0\),
\[
\bigl\|\dot{\mathbf{x}}^*(t;x)\bigr\|
\le \max_{y\in\cl{\Omega_{v^*}}} \left( \|f(y)\|
   + \tfrac12 \|R^{-1}\|\,
      \|g(y)\|^{2}\,
      \|\nabla v^{*}(y)\|\right)
=: C < \infty,
\]
where finiteness follows from the compactness of \(\cl{\Omega_{v^*}}\) and the
continuity of \(f\), \(g\), and \(\nabla v^{*}\). Hence
\(t\mapsto \mathbf{x}^*(t;x)\) is Lipschitz with constant \(C\). Since \(h\in C^{1}(\mathbb{R}^N,\mathbb{R})\) and
\(\nabla h\) is bounded on \(\cl{\Omega_{v^*}}\), the mean value theorem implies, for
all \(t_1,t_2\ge 0\),
\[
\bigl|p(t_1)-p(t_2)\bigr|
= \bigl|h\!\bigl(\mathbf{x}^*(t_1;x)\bigr)-h\!\bigl(\mathbf{x}^*(t_2;x)\bigr)\bigr|
\le \Bigl(\max_{y\in\cl{\Omega_{v^*}}}\|\nabla h(y)\|\Bigr)\, C\,|t_1-t_2|,
\]
so \(p\) is uniformly continuous on \([0,\infty)\). Combining this with
\eqref{eq:finite-cost} and applying Barbalat’s lemma gives
\[
\lim_{t\to\infty} h\!\bigl(\mathbf{x}^*(t;x)\bigr)=0 .
\]
Finally, since \(h\) is continuous and positive definite, and \(\cl{\Omega_{v^*}}\) is compact, \eqref{eq:asymEq} follows; this can be established easily via a proof by contradiction.

 \medskip
 \noindent 
\textbf{Step 4 ($v(x)\ge v^*(x)$ for $x\in\Omega_v$).}
As mentioned at the beginning of the proof, everything established for $v^*$ also holds for $v$,
since $v$ is positive semidefinite and satisfies the HJB equation by assumption.
In particular, there exists a forward-invariant set $\Omega_v\subset\Omega$ for the closed-loop dynamics \eqref{eq:clCOVF}
and, as in Step~3,
\begin{align}\label{eq:asymEq2}
\lim_{t\to\infty}\|\mathbf{x}_{u_v}(t;x)\|=0 \quad \text{for all } x\in\Omega_v .
\end{align}
For any $x\in\Omega_v$, the HJB identity along the trajectory
$t\mapsto \mathbf{x}_{u_v}(t;x)$ gives, for every $T>0$, by integrating
\[
v(x)
= v \bigl(\mathbf{x}_{u_v}(T;x)\bigr)
  + \inte{0}{T}{t}{
      h\!\bigl(\mathbf{x}_{u_v}(t;x)\bigr)
      + \|u_v\!\bigl(\mathbf{x}_{u_v}(t;x)\bigr)\|_{R}^{2}
    } .
\]
Letting $T\to\infty$ and using \eqref{eq:asymEq2} together with the continuity and
positive definiteness of $v$ (so that $v(\mathbf{x}_{u_v}(T;x))\to v(0)=0$), we obtain
\[
v(x)
= \inte{0}{\infty}{t}{
      h\!\bigl(\mathbf{x}_{u_v}(t;x)\bigr)
      + \|u_v\!\bigl(\mathbf{x}_{u_v}(t;x)\bigr)\|_{R}^{2}
    }
= J_{\infty}(x,u_v)
\;\ge\; \inf_{\mathbf{u}\in\mathcal{U}_{\infty}} J_{\infty}(x,\mathbf{u})
= v^*(x),
\]
 since  \(u_{v}\!\bigl(\mathbf{x}_{u_v}(\,\cdot\,;x)\bigr) \in \mathcal{U}_{\infty}\) by
 \begin{align*}
     \sup_{t \geq 0} \Vert u_{v}\!\left(\mathbf{x}_{u_v}(t;x) \right) \Vert  \leq \max_{x \in \cl{\Omega_v}} \Vert  \tfrac12\,R^{-1} g(x)^{\!\top}\nabla v(x)\Vert.
 \end{align*}
This establishes \(v(x)\ge v^{*}(x)\).

 \medskip
 \noindent 
\textbf{Step 5 ($v(x)\le v^*(x)$ for $x\in\Omega_{v^*}$).}
Define the Hamiltonian
\[
H_v(x,u)
:= \big\langle \nabla v(x),\, f(x)+g(x)\,u \big\rangle + h(x) + \|u\|_{R}^{2}.
\]
Since $R$ is positive definite, the map $u\mapsto H_v(x,u)$ is strictly convex, and differentiation with
respect to $u$ gives the unique minimizer
\[
u_v(x) \;=\; -\tfrac{1}{2}\,R^{-1}\,g(x)^{\!\top}\nabla v(x).
\]
Because $v$ satisfies the HJB equation on $\Omega$,
\begin{align}\label{eq:minHami}
    0 \;=\; \min_{u \in \mathbb{R}^M} H_v(x,u) \;=\; H_v\!\bigl(x,u_v(x)\bigr)
\;\le\; H_v(x,u)\qquad\text{for all } x\in\Omega,\; u\in \mathbb{R}^M.
\end{align}
Fix $x\in\Omega_{v^*}$ and evaluate this inequality along the optimal closed-loop
trajectory $t\mapsto \mathbf{x}^*(t;x)$ (which remains in $\Omega_{v^*}$ by Step~2) with
the control $u=u^*\!\bigl(\mathbf{x}^*(t;x)\bigr)$. Using
\[
\frac{\text{d}}{\text{d}t}\,v\!\bigl(\mathbf{x}^*(t;x)\bigr)
= \big\langle \nabla v(\mathbf{x}^*(t;x)),\, f(\mathbf{x}^*(t;x))
    + g(\mathbf{x}^*(t;x))\,u^*\!\bigl(\mathbf{x}^*(t;x)\bigr) \big\rangle,
\]
we obtain by \eqref{eq:minHami}, for all $t\ge 0$,
\[
0 \;\le\; \frac{\text{d}}{\text{d}t}\,v \bigl(\mathbf{x}^*(t;x)\bigr)
          + h\!\bigl(\mathbf{x}^*(t;x)\bigr)
          + \bigl\|u^*\!\bigl(\mathbf{x}^*(t;x)\bigr)\bigr\|_{R}^{2}.
\]
Integrating from $0$ to $T>0$ yields
\[
0 \;\le\; v \bigl(\mathbf{x}^*(T;x)\bigr) - v(x)
          + \inte{0}{T}{t}{
                h\!\bigl(\mathbf{x}^*(t;x)\bigr)
              + \bigl\|u^*\!\bigl(\mathbf{x}^*(t;x)\bigr)\bigr\|_{R}^{2}
            }.
\]
By Step~3, $\mathbf{x}^*(T;x)\to 0$ as $T\to\infty$; since $v$ is continuous and
positive semidefinite, $v\bigl(\mathbf{x}^*(T;x)\bigr)\to v(0)=0$. Letting $T\to\infty$ gives
\[
v(x) \;\le\; \inte{0}{\infty}{t}{
                h\!\bigl(\mathbf{x}^*(t;x)\bigr)
              + \bigl\|u^*\!\bigl(\mathbf{x}^*(t;x)\bigr)\bigr\|_{R}^{2}
            }
= J_{\infty}\!\bigl(x,u^*\bigr) \;=\; v^*(x).
\]
Thus $v(x)\le v^*(x)$ for all $x\in\Omega_{v^*}$.

\medskip
\noindent
\textbf{Step 6 (Conclusion).}
Combining Step~5, which establishes \(v(x)\le v^*(x)\) for \(x\in\Omega_{v^*}\), with
Step~4, which shows \(v(x)\ge v^*(x)\) for \(x\in\Omega_v\), we conclude that
$
v(x)=v^*(x)$ for all $x\in\tilde{\Omega}:=\Omega_{v^*}\cap\Omega_v.
$
Since both \(\Omega_{v^*}\) and \(\Omega_v\) contain neighborhoods of the origin, their
intersection \(\tilde{\Omega}\) also contains a neighborhood of the origin.

\subsection{Proof of Corollary \ref{coro:VerificAnalytic}} \label{sec:appoptVerAna}
\begin{proof}
Recall the identity theorem for real-analytic functions \cite{krantz2002primer}:  
if $\Omega\subset\mathbb{R}^N$ is open and connected and if 
$f,g\in A(\Omega)$ coincide on a nonempty open set $\Omega'\subset\Omega$, then
$f\equiv g$ on all of~$\Omega$.\\
By Theorem \ref{thm:VerOpt} there exists a  set
$\tilde{\Omega}\subset\Omega$ that contains a neighborhood of the origin and on
which $v(x)=v^*(x)$ for every $x\in\tilde{\Omega}$.  
This neighborhood plays the role of $\Omega'$ in the identity theorem.  Hence,
by that theorem and the analyticity of $v$ and~$v^*$, we conclude that
$
  v(x)=v^*(x)$ for all $x\in\Omega.
$
\end{proof}
\subsection{Proof of Corollary \ref{coro:VerificConstant}} \label{sec:VerificConstant}
\begin{proof}
The argument follows the strategy of the proof of Theorem~\ref{thm:VerOpt}; the only change is the definition of the set~$\Omega_v$.
Fix $r>0$ such that the closed ball $\cl{B_r(0)}\subset\Omega$, and set
\[
  c \;:=\; \alpha r^{2},
  \qquad
  \Omega_v \;:=\; \bigl\{\, y\in\cl{B_r(0)} \;\big|\, v(y)<c \bigr\}.
\]
Since $\min_{\{y\in\Omega:\,\|y\|=r\}} v(y)\ge \alpha r^2 = c$, the same arguments as in Step~2 of the proof of Theorem~\ref{thm:VerOpt} imply that $\Omega_v$ is forward invariant for the dynamics~\eqref{eq:clCOVF}. Exactly as in the proof of Theorem~\ref{thm:VerOpt}, one then shows that
\[
  v(x) \;=\; v^{*}(x)\qquad\text{for all }x\in \Omega_v \cap \Omega_{v^{*}}.
\]
Moreover, since $v$ satisfies $v(x)\le \beta\|x\|^{2}$, we have for every
$x\in B_{\sqrt{\alpha/\beta}\,r}(0)$ that
\[
  v(x)\;\le\; \beta\|x\|^{2}
  \;<\; \beta\,\frac{\alpha}{\beta}\,r^{2}
  \;=\; \alpha r^{2}
  \;=\; c,
\]
and hence $B_{\sqrt{\alpha/\beta}\,r}(0)\subset \Omega_v$.
Define
\[
  \widetilde{\Omega}
  \;:=\;
  B_{\sqrt{\alpha/\beta}\,r}(0)\,\cap\,\Omega_{v^{*}}
  \;\subset\; \Omega_v \cap \Omega_{v^{*}}.
\]
By the preceding argument, $v=v^{*}$ on $\widetilde{\Omega}$.
Thus we have found the desired set $\widetilde{\Omega}$, which depends only on~$\alpha$, $\beta$, and~$v^{*}$, and which contains a neighborhood of the origin.
\end{proof}

\section{Proof of Theorem \ref{theo:quadKernel}}\label{sec:quadKernelProof}
\begin{proof}
The kernel $\langle x,y \rangle^2$ is p.d.; thus, by \cite{aronszajn1950theory}, $k$ is p.d. as a product of two p.d. kernels.

\medskip
\noindent Next, for the linear independence, assume, to the contrary, that there exist pairwise distinct points $x_1,\dots,x_n\in\mathbb{R}^N\setminus\{0\}$ and coefficients
$a_i\in\mathbb{R}$ and $b_{i,s}\in\mathbb{R}$, not all zero, such that the functional
\[
\sum_{i=1}^n a_i\,\delta_{x_i}(\,\cdot\,)\;+\;\sum_{i=1}^n\sum_{s=1}^N b_{i,s}\,\delta_{x_i}(\,\cdot\,)  \circ\partial^s(\,\cdot\,)
\]
vanishes in $\HkdualR$. By the Riesz representation theorem, its representer is
\[
g(\cdot)
=\sum_{i=1}^n a_i\,k(x_i,\cdot)
+\sum_{i=1}^n\sum_{s=1}^N b_{i,s}\,\partial_{1}^{s}k(x_i,\cdot)
\in\HkR,
\]
and $\|g\|_{\HkR}^2=0$. Expanding this norm using the reproducing identities \eqref{eq:repProp} and \eqref{eq:repPropDeriv} yields
\begin{align*}
0
=\sum_{i,j=1}^n a_i a_j\,k(x_i,x_j)
+\sum_{i,j=1}^n \sum_{s=1}^N a_i b_{j,s}\,\bigl(\partial^s_{2}k(x_i,x_j)+\partial^s_{1}k(x_j,x_i)\bigr) 
+\sum_{i,j=1}^n\sum_{s,\ell=1}^N b_{i,s} b_{j,\ell}\,\partial^s_{1}\partial_2^{\ell}k(x_i,x_j).
\end{align*}
Next we exploit the Fourier representation of $\tilde k$. Since $\phi\in L^1(\mathbb{R}^N,\mathbb{R})$, by Bochner's theorem
\[
\phi(z)=\frac{1}{(2\pi)^N}\int_{\mathbb{R}^N} \left [ e^{i\langle w,z\rangle}\,\widehat{\phi}(w) \right ] \text{d}w,
\qquad \widehat{\phi}\ge 0,
\]
and $\widehat{\phi}$ is continuous. Moreover, $\tilde k$ is strictly positive definite, so there exists a nonempty open set $U\subset\mathbb{R}^N$ with $\widehat{\phi}(w)>0$ for all $w\in U$ (see \cite[Theorem 6.11]{wendland2004}).
Introduce the complex Frobenius inner product $\langle A,B\rangle_{\mathcal{F}}=\operatorname{trace}(\overline{A}^{\top}B)$ and the matrix–valued maps
\[
A(x,w):=x x^{\top}\,e^{-i\langle w,x\rangle},\qquad
B_s(x,w):=\bigl(x e_s^{\top}+e_s x^{\top}-i\,w_s\,x x^{\top}\bigr)e^{-i\langle w,x\rangle},
\]
where $e_s$ is the $s$-th unit vector and $w_s$ the $s$-th coordinate of $w$.
A direct computation shows
\begin{align}
k(x,y)
&=\frac{1}{(2\pi)^N}\int_{\mathbb{R}^N} \left [ \bigl\langle A(x,w),A(y,w)\bigr\rangle_{\mathcal{F}}\,\widehat{\phi}(w) \right ] \text{d}w, \label{eq:rep1}\\
\partial^s_{1 }k(x,y)
&=\frac{1}{(2\pi)^N}\int_{\mathbb{R}^N} \left [ \bigl\langle B_s(x,w),A(y,w)\bigr\rangle_{\mathcal{F}}\,\widehat{\phi}(w)\right ]\text{d}w, \label{eq:rep2}\\
\partial^s_{1 }\partial_2^{\ell}k(x,y)
&=\frac{1}{(2\pi)^N}\int_{\mathbb{R}^N}\left [  \bigl\langle B_s(x,w),B_\ell(y,w)\bigr\rangle_{\mathcal{F}}\,\widehat{\phi}(w)\right ]\text{d}w. \label{eq:rep3}
\end{align}
The interchange of differentiation and integration in \eqref{eq:rep2}–\eqref{eq:rep3} is justified by the dominated convergence theorem under the integrability assumption
$
(1+\|w\|^{2})\,\widehat{\phi}(w)\in L^{1}(\mathbb{R}^{N}),
$
which provides an $L^{1}$-dominating envelope for the integrands.
Inserting \eqref{eq:rep1}–\eqref{eq:rep3} into the expansion of $\|g\|_{\HkR}^2$ and regrouping the integrand, we obtain
\[
0=\frac{1}{(2\pi)^N}\int_{\mathbb{R}^N} \bigl\| C(w)\bigr\|_{\mathcal{F}}^{2}\,\widehat{\phi}(w)\,dw,
\]
where
\[
C(w):=\sum_{i=1}^n a_i\,A(x_i,w)
\;+\;\sum_{i=1}^n\sum_{s=1}^N b_{i,s}\,B_s(x_i,w).
\]
Because $\widehat{\phi}\ge 0$ and $\widehat{\phi}>0$ on the open set $U$, it follows by continuity that $C(w)=0$ for all $w\in U$.\\
Since $x_1,\dots,x_n$ are pairwise distinct and $U$ is open, there exists an open set $\tilde U\subset U$ such that the scalars
$
\alpha_{i,v}:=\langle v,x_i\rangle
$
are pairwise distinct for every $v\in\tilde U$. Indeed, one can take
\[
\tilde U \;:=\; U \setminus \bigcup_{1\le i<j\le n}\Bigl\{y\in\mathbb{R}^N:\,\langle y,\,x_i-x_j\rangle=0\Bigr\},
\]
the complement of a finite union of closed hyperplanes, hence open. Fix $v\in\tilde U$ and abbreviate $\alpha_i:=\langle v,x_i\rangle$. Since $\tilde U$ is open and $v\neq 0$, there exists $\delta>0$ such that $t v\in\tilde U$ for all $t\in[1-\delta,\,1+\delta]$. For such $t$ we have $C(t v)=0$, i.e.
\[
\sum_{i=1}^n e^{-i t \alpha_i}\!\left[\,a_i\,x_i x_i^{\top}
+\sum_{s=1}^N b_{i,s}\,\bigl(x_i e_s^{\top}+e_s x_i^{\top}\bigr)\right]
\;-\; i t \sum_{i=1}^n e^{-i t \alpha_i}\!\left[\left(\sum_{s=1}^N b_{i,s}\,v_s\right) x_i x_i^{\top}\right]
\;=\;0 .
\]
Thus the matrix–valued trigonometric polynomial
\[
t \longmapsto \sum_{i=1}^n \bigl(M_i - i t\,N_i\bigr)\,e^{-i t \alpha_i},
\qquad
M_i:=a_i\,x_i x_i^{\top}+\sum_{s=1}^N b_{i,s}\,(x_i e_s^{\top}+e_s x_i^{\top}),\quad
N_i:=\Bigl(\sum_{s=1}^N b_{i,s} v_s\Bigr)\,x_i x_i^{\top},
\]
vanishes on an interval. Since the frequencies $\alpha_1,\dots,\alpha_n$ are distinct, the scalar functions
$
\bigl\{\,e^{-i t \alpha_i},\; t\,e^{-i t \alpha_i}\,\bigr\}_{i=1}^n
$
are linearly independent on any interval (see \cite{heittokangas2023value}). Hence $M_i=0$ and $N_i=0$ for each $i=1,\dots,n$. Because $x_i\neq 0$, the equality $N_i=0$ implies $\sum_{s=1}^N b_{i,s} v_s=0$. As $v$ can be chosen arbitrarily from the open set $\tilde U$, it follows that $b_{i,s}=0$ for all $i$ and all $s$. Then $M_i=0$ reduces to $a_i\,x_i x_i^{\top}=0$, whence $a_i=0$ for all $i$. This contradicts the assumption that not all coefficients are zero. Therefore, the family \eqref{eq:fam} is linearly independent in $\HkdualR$.

\end{proof}

\end{document}